\documentclass[11pt]{article}
\usepackage{amsfonts}
\usepackage{amsmath}
\usepackage{amsthm}
\usepackage{amssymb}
\usepackage{mathrsfs}
\usepackage[numbers]{natbib}
\usepackage[fit]{truncate}
\usepackage{fullpage}
\usepackage{mathtools}
\usepackage{makecell}
\usepackage{subfig}
\usepackage{hyperref}
\usepackage{dsfont}
\usepackage{xcolor}

\newcommand{\N}{\mathbb{N}}
\newcommand{\R}{\mathbb{R}}
\renewcommand{\P}{\mathbb{P}}
\newcommand{\E}{\mathbb{E}}

\theoremstyle{plain}
\newtheorem{theorem}{Theorem}[section]
\newtheorem{corollary}[theorem]{Corollary}
\newtheorem{lemma}[theorem]{Lemma}

\newtheorem{proposition}[theorem]{Proposition}

\newtheorem{conjecture}[theorem]{Conjecture}
\theoremstyle{definition}

\newtheorem{example}[theorem]{Example}

\begin{document}

\title{Tails of explosive birth processes\\ and applications to non-linear Pólya urns}
\author{Thomas Gottfried and Stefan Großkinsky}
\date{\today}
\maketitle

\begin{abstract}
We derive a simple expression for the tail-asymptotics of an explosive birth process at a fixed observation time conditioned on non-explosion. Using the well-established exponential embedding, we apply this result to compute the tail distribution of the number of balls of a losing colour in generalized Pólya urn models with super-linear feedback, which are known to exhibit a strong monopoly for the winning colour where losers only win a finite amount. Previous results in this direction were restricted to two colours with the same feedback, which we extend to an arbitrary finite number of colours with individual feedback mechanisms. 
As an apparent paradox, losing colours with weak feedback are more likely to win in many steps than those with strong feedback. Our approach also allows to characterize the correlations of several losing colours, which provides new insight in the distribution of other related quantities like the total number of balls for losing colours or the time to monopoly. In order to give a complete picture, we also consider sub-linear feedback and discuss asymptotics for a diverging number of colours.
\end{abstract}

\section{Introduction}

Birth processes are fundamental models for growth processes, having applications in numerous fields such as biology \cite{birthBio}, chemistry \cite{BirthChem}, computer science \cite{BirthInfo} or reliability theory \cite{shock}. The simple birth process was first introduced by Yule \cite{yule} in 1924 in connection with evolutionary theory and independently in 1937 by Furry \cite{furry} in a physical context, which is why it is also known as Yule-Furry process. Birth and death processes, which were introduced by Feller \cite{feller} in 1939, pose an important and well-studied extension of this model.

Given a \textit{rate function} $F\colon\N\coloneqq\{1, 2,\ldots\}\to(0, \infty)$,  a (pure) birth process is a homogeneous, continuous-time Markov process $(\Xi(t))_{t\ge0}$ on the state space $\N$, which jumps from $k$ to $k+1$ at rate $F(k)$. More formally, take independent random variables $\tau(k)$, which are exponentially distributed with parameter $F(k)$. Then a birth process with initial condition $\Xi(0)\in\N$ is the corresponding counting process with sojourn times $\tau(k)$, i.e.
\begin{equation}\label{eq: bp}
\Xi(t)\coloneqq \min\left\{k\in\N\colon \sum_{l=\Xi(0)}^k\tau(l) > t\right\}\quad\text{for all }t\in[0, \infty)\,,
\end{equation}
where $\min\emptyset\coloneqq\infty$. A widely studied case is the \textit{simple} (or linear) birth process, where $F(k)=\lambda k$ for some $\lambda>0$. $F(k)=\lambda$ corresponds to a homogeneous Poisson process.

Since $F$ is strictly positive we have $\Xi (t)\to\infty$ as $t\to\infty$ almost surely. $\Xi$ is called \textit{explosive} if $\Xi(t)=\infty$ for some finite $t<\infty$, which is obviously equivalent to $\sum_{l=\Xi(0)}^\infty\tau(k)<\infty$ and we define the explosion time
\begin{equation}\label{eq: defT}
T\coloneqq\sum_{k=\Xi(0)}^\infty\tau(k)<\infty\quad\text{with density }g\,.
\end{equation}
The well-established Theorem of Feller-Lundberg \cite{feller2, lundberg} states that
 \begin{equation}\label{eq: explosion}
     \Xi\mbox{ is explosive (i.e. $T<\infty$) if and only if}\quad\sum_{k=1}^\infty\frac{1}{F(k)}<\infty\,.
 \end{equation}
\cite{feller2} further contains an explicit formula for the distribution of $\Xi(t)$ for fixed $t$,
\begin{equation}\label{eq: feller}
     \P(\Xi(t)=x)=\sum_{k=\Xi(0)}^x\frac{\prod_{l=\Xi(0)}^{x-1}F(l)}{\prod_{l=\Xi(0)\atop l\ne k}^x(F(l)-F(k))}e^{-F(k)t}\quad\text{ for }x\in\N\,.
\end{equation}
Remarkably, this holds for both explosive and non-explosive birth processes, and for the non-explosive case, \cite{orsingher} provides a formula for the expectation of $\Xi(t)$. For linear birth processes, $\Xi(t)$ has a negative binomial distribution and grows exponentially in $t$ (see e.g. \cite{bartlett, yule2}). Another result worth mentioning in the context of explosive birth processes is \cite[Theorem 3.1]{pakes}, which determines the asymptotics of $\Xi(T-t)$ for $t\to0$. The asymptotics for $t\to\infty$ of non-explosive birth processes are described in e.g. \cite{waugh, waugh2, Barbour}.\\

Section \ref{sec: birth} of this paper is dedicated to the explosive case, where we characterize the tail of $\Xi(t)$ for fixed $t>0$ conditioned on not having exploded yet, i.e. $T>t$. This question was addressed in a recent paper \cite{forbes2} based on (\ref{eq: feller}) and focusing on numerical studies, and we choose a different approach to get a rigorous and fairly general result (Theorem \ref{thm: birthMain}). \\

A related model for reinforced growth of interacting agents in discrete time is the generalized Pólya urn model, which was initially proposed in \cite{Hill}. Consider an urn with balls of $A\in\{2, 3,\ldots\}$ different colours, where each colour $i\in[A]\coloneqq\{1,\ldots, A\}$ is equipped with a feedback (or weight) function $F_i\colon \N\to(0, \infty)$. Denote by $X_i$ the number of balls of colour $i$. A colour is selected from the urn with probability proportional to $F_1(X_1),\ldots, F_A(X_A)$ and a new ball of that colour is added to the urn. The generalized Pólya urn process iterates this procedure, defining a discrete-time Markov process $(X(n))_{n\in\N_0}=((X_1(n),\ldots,X_A(n)))_{n\in\N_0}$ on the state space $\N^A$. For a comprehensive overview on this model, see \cite{wir} and references therein. If at least one $F_i$ satisfies (\ref{eq: explosion}), then this process exhibits strong monopoly, i.e. all but one colour are drawn only finitely often. Using our results from Section \ref{sec: birth} and a standard exponential embedding, we determine the tail distribution of the number of balls of a losing colour (or agent) in Section \ref{sec: applicatonPolya} (in particular Theorem \ref{thm: loserSuperlin}), which generalizes known results in \cite{Zhu, Oliveira, Cotar} to situations with more than two agents with possibly different feedback functions. Moreover, we characterize the tail dependence of several losing colours and discuss applications, which encompass e.g. the total wealth of all losers or the time to monopoly. What may seem paradoxical at first is that losers with feedback close to the transition (\ref{eq: explosion}) are most likely to win in many steps. In order to provide a detailed description of this transition, we finally also consider the wealth of losers in systems, where not all agents satisfy (\ref{eq: explosion}) (Section \ref{sec: LoserSublin}).\\
 
\noindent For later convenience we establish the following notation for real sequences $(x_k)$ and $(y_k)$,
\begin{align}
    x_k\sim y_k&\quad\Leftrightarrow\quad \lim_{k\to\infty}\frac{x_k}{y_k}=1\,,\nonumber\\
    x_k\prec y_k&\quad\Leftrightarrow\quad \limsup_{k\to\infty}\frac{x_k}{y_k}<\infty\,\nonumber\\
    x_k\asymp y_k&\quad\Leftrightarrow\quad x_k\prec y_k\text{ and }y_k\prec x_k \ .
\end{align}

\section{Tails of explosive birth processes}\label{sec: birth}

\subsection{Main result}

With notation from the introduction, we can directly formulate our main result on the tail of birth processes conditioned on non-explosion.

\begin{theorem}\label{thm: birthMain}
    Assume that (\ref{eq: explosion}) is fulfilled. Then 
    \begin{equation*}
        \P(\Xi(t)=x)\sim  \frac{g(t)}{F(x)}\quad\text{for }x\to\infty\,,
    \end{equation*}
     holds for all $t>0$. For any $t_0>0$, the convergence is uniform in $t>t_0$. Consequently, we get the following conditional tail distribution:
    \begin{equation*}
        \P(\Xi(t)> x\,|\,T>t)\sim -\frac{d}{dt}\log(\P(T>t))\sum_{k=x+1}^\infty\frac{1}{F(k)}\quad\text{for }x\to\infty
    \end{equation*}
\end{theorem}

\begin{figure}
  \centering
  \subfloat[$F(k)=k^2$]{\includegraphics[width=0.5\linewidth]{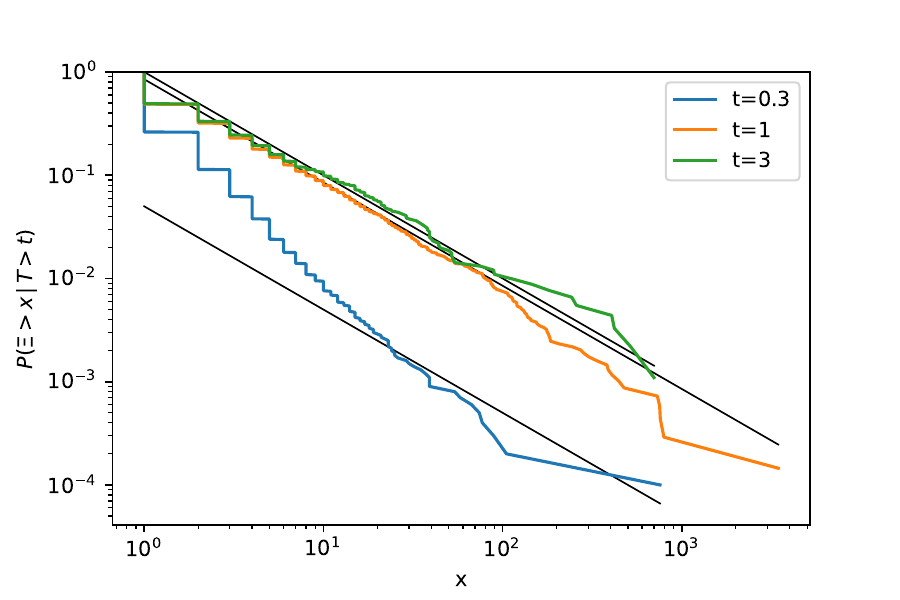}}
  \subfloat[$F(k)=e^{k-1}$]{\includegraphics[width=0.5\linewidth]{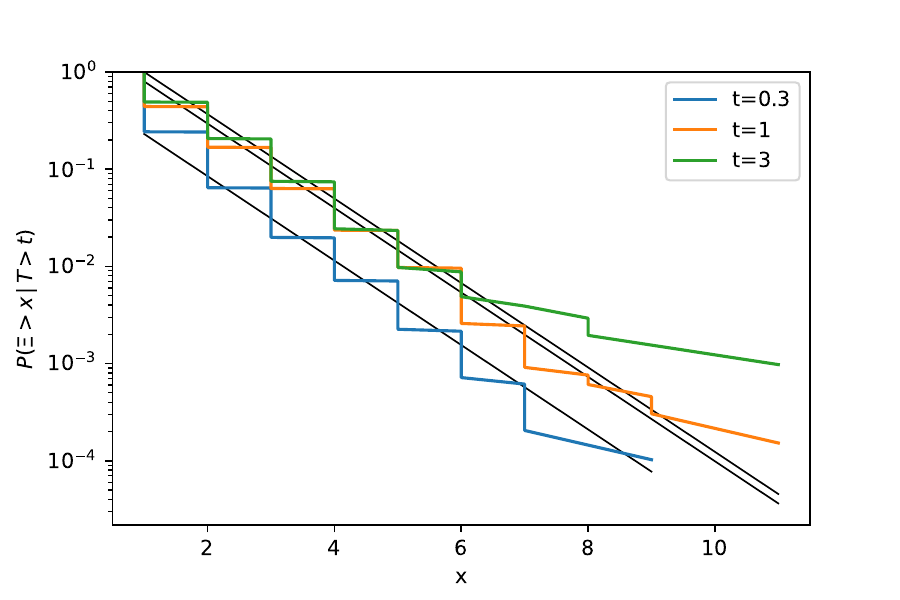}}
  \caption{The empirical distribution of $\Xi(t)$ conditioned on $T>t$ for different $F$ and $t$. The black lines show the predicted tail according to Theorem \ref{thm: birthMain}. 10,000 realizations of the process \eqref{eq: bp} were simulated each. In (a) 10 (resp. 3112, resp. 9088) have already exploded at time $t=0.3$ (resp. $t=1$, resp. $t=3$) and in (b) 30 (resp. 3424, resp. 8974). As the criterion for explosion the simulation was stopped after $10^6$ summands in (a) and 100 summands in (b).}
  \label{figure: birthSimulation}
\end{figure}

Figure \ref{figure: birthSimulation} numerically illustrates Theorem \ref{thm: birthMain} for quadratic and exponential rate functions. Indeed, when $t$ is moderate or large, then the approximation seems good even for small $x$. The approximation is rather inaccurate for small $t$, when $\Xi(t)$ is likely to be close to $\Xi(0)$. Theorem \ref{thm: birthMain} separates the tail of $\Xi(t)$ into a prefactor that depends only on $t$ and a sequence that depends only on $x$. The prefactors $-\frac{d}{dt}\log\P(T>t)=\frac{g(t)}{\P(T>t)}$ are plotted in Figure \ref{figure: ParetoFaktor} for several feedback functions and will be discussed in the following.

 \begin{figure}
  \centering
  \includegraphics[width=0.7\linewidth]{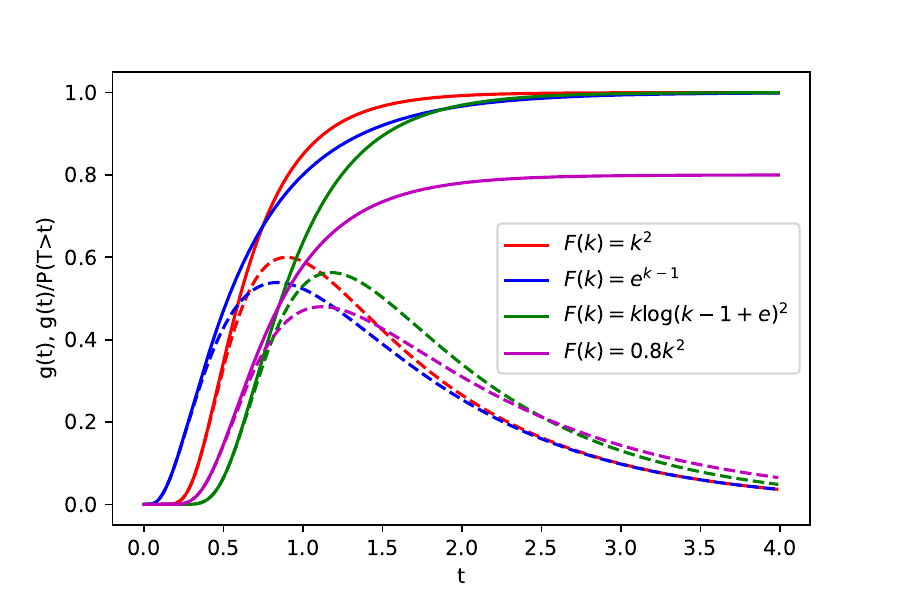}
  \caption{ The factors $-\frac{d}{dt}\log\P(T>t)=\frac{g(t)}{\P(T>t)}$ (full lines)  and $g(t)$ (dotted lines) from Theorem \ref{thm: birthMain} for different feedback functions. The full lines converge to $F(1)$ for $t\to\infty$, as discussed later in \eqref{eq: prelim}. For the numerical computation, $g$ was approximated by $g_n$ from Lemma \ref{lemma: zhu} with $n=100$.}
  \label{figure: ParetoFaktor}
\end{figure}

In the proof of Theorem \ref{thm: birthMain}, we will be particularly interested in first passage times, which we formally define as
 \begin{equation*}
     T(k)\coloneqq \sum_{l=\Xi(0)}^k\tau(l)\quad\text{with density }g_k
 \end{equation*}
 and the remaining time until explosion
 \begin{equation*}
     \bar T(k)\coloneqq \sum_{l=k+1}^\infty\tau(l)\quad\text{with density }\bar g_k
 \end{equation*}
 for $k\ge \Xi(0)$. We start the proof with an auxiliary lemma on analytic properties of the density $g$.

\begin{lemma}\label{lemma: gProperties}
    If (\ref{eq: explosion}) is satisfied, then the density $g$ of the explosion time $T$ (\ref{eq: defT}) exists and has the following properties.
    \begin{enumerate}
        \item $g$ is bounded with $g(0)=0$ and $g(t)>0$ for all $t>0$.
        \item There is a unique maximum $m>0$, such that $g(t)$ is increasing for $t<m$ and decreasing for $t>m$.
        \item For any $t_0>0$, the logarithmic derivative $\big|\frac{d}{dt}\log g(t)\big|$ is bounded uniformly in $t\ge t_0$.
        \item $g$ is not analytic in $0$. In particular, $\frac{d^k}{dt^k}g(0)=0$ holds for all $k\ge 1$
    \end{enumerate}
\end{lemma}

\begin{proof}
    For simplicity of notation, assume $\Xi(0)=1$. Existence of $g$ follows e.g. from the inversion formula for characteristic functions. Since $g_2(0)=0$ and $\P(T(2)\le x)\ge\P(T\le x)$, we have $g(0)=0$. From the  convolution formula for $T(1)+\bar T (1)$, i.e.
    \begin{align}\label{eq: g_convolution}
        g(t)=\int_0^t g_1(t-s)\bar g_1(s)\,ds=F(1)\int_0^t e^{-F(1)(t-s)}\bar g_1(s)\,ds\,,
    \end{align}
   we can conclude that $g$ is bounded by $\max_{t\ge0}g_1(t)=F(1)$. From Markov's inequality, we get that for any $\epsilon>0$ and $k$ large enough
   $$\P(\bar T(k)<\epsilon/2)>0\,,$$
   such that $\P(T<\epsilon)>0$ holds, too. Hence, $g(t)>0$ for small $t>0$ and analogously $\bar g_1(t)>0$ for small $t>0$. Then $g(t)>0$ for any $t>0$ follows from the convolution formula (\ref{eq: g_convolution}).

    For 2., we show that for all $k\in\N$ there is $m_k\ge0$ such that $g_k'(t)>0$ for all $t<m_k$ and $g_k'(t)<0$ for all $t>m_k$. This implies 2. since $g_k(t)\to g(t)$ for $k\to\infty$ uniformly in $t\ge0$ (see \cite{boos} for a rigorous argument). We show this by induction over $k$. Assume that $g_k$ has the desired properties. Then we get from the convolution formula for $T(k+1)=T(k)+\tau(k+1)$
    \begin{equation}\label{eq: inter1}
        g_{k+1}'(t)=\int_0^t F(k+1) e^{-F(k+1) s}g_k'(t-s)ds+  F(k+1)e^{-F(k+1) t}g_k'(0)\,.
    \end{equation}
    For small enough $t>0$, it follows $g_{k+1}'(t)>0$ since $g_{k}'(t)>0$. Define
    \begin{equation}\label{eq: inter2}
        m_{k+1}\coloneqq \min \{t>0\colon g_{k+1}'(t)=0\}
    \end{equation}
    and rephrase
    \begin{align*}
        \frac{g_{k+1}'(t)}{F(k+1)}&=\int_{0}^{t-m_{k+1}} e^{-F(k+1) s}g_k'(t-s)ds+\int_{t-m_{k+1}}^{t} e^{-F(k+1) s}g_k'(t-s)ds+e^{-F(k+1) t}g_k'(0)\\
     &=\int_{0}^{t-m_{k+1}} e^{-F(k+1) s}g_k'(t-s)ds\\
     &\quad+e^{-F(k+1) (t-m_{k+1})}\left(\int_{0}^{m_{k+1}} e^{-F(k+1) s}g_k'(m_{k+1}-s)ds+  e^{-F(k+1) m_{k+1}}g_k'(0)\right)\\
     &=\int_{0}^{t-m_{k+1}} e^{-F(k+1) s}g_k'(t-s)ds
    \end{align*}
     for $t>m_{k+1}$, using \eqref{eq: inter1} and \eqref{eq: inter2}. Since necessarily $g_k'(t)<0$ for $t>m_{k+1}>m_k$, we get $g_{k+1}'(t)<0$ for $t>m_{k+1}$.

     For 3., it only remains to show $\liminf_{t\to\infty}\frac{d}{dt}\log g(t)>-\infty$ together with 1. and 2.. This follows from
     \begin{align*}
         \frac{d}{dt}\log g(t)=\frac{g'(t)}{g(t)}=\frac{-F(1)^2 \int_0^t e^{-F(1) (t-s)}\bar g_1(s)ds + F(1)\bar g_1(t) }{ \int_0^t F(1)e^{-F(1) (t-s)}\bar g_1(s)ds }>-F(1)\,.
     \end{align*}

     For 4., we first show that $\frac{d^k}{dt^k}g_{k+2}(0)=0$  via induction over $k\ge 1$. Assume that $\frac{d^{k-1}}{dt^{k-1}}g_{k+1}(0)=0$, which is easy to check for $k=1$. Then:
     \begin{align*}
         \frac{d^k}{dt^k}g_{k+2}(t)&=F(k+1)\frac{d}{dt}\int_0^t e^{-F(k+1)s}\frac{d^{(k-1)}}{dt^{(k-1)}} g_{k+1}(t-s)ds\\
         &=F(k+1)\frac{d}{dt}\int_0^t e^{-F(k+1)(t-s)}\frac{d^{(k-1)}}{dt^{(k-1)}} g_{k+1}(s)ds\\
         &=-F(k+1)^2\int_0^t e^{-F(k+1)(t-s)}\frac{d^{(k-1)}}{dt^{(k-1)}} g_{k+1}(s)ds+F(k+1)\frac{d^{(k-1)}}{dt^{(k-1)}} g_{k+1}(t)\xrightarrow{t\to 0}0
     \end{align*}
Now, assume that there is $k\ge1$ such that $\frac{d^k}{dt^k}g(0)>0$. Then $g_{k+2}(t)<g(t)$ for small enough $t>0$ since $\frac{d^k}{dt^k}g_{k+2}(0)=0$. On the other hand,  $g_{k+2}(t)> g(t)$ for small enough $t>0$ due to $\P(T(k+2)\le t)>\P(T\le t)$. Hence,  $\frac{d^k}{dt^k}g(0)=0$.
\end{proof}


For injective $F\colon \N\to(0, \infty)$, \cite{Zhu} provides a useful explicit formula for $g(t)$, which is useful for numerical approximation (see Figure \ref{figure: ParetoFaktor}).

\begin{lemma}\label{lemma: zhu}\cite[Lemma 3.2.1., Lemma 3.2.2.]{Zhu}
    Let $F(1),\ldots ,F(n)$ be positive and pairwise distinct with $\tau(k)$ exponentially distributed and independent with parameters $F(k)$. Then $T(n)=\sum_{k=1}^n \tau(k)$ has the density function
    $$g_n(x)=\left(\prod_{k=1}^n F(k)\right)\sum_{k=1}^n\frac{e^{-F(k)x}}{\prod_{l\ne k}(F(l)-F(k))}\quad\text{for }x\ge0\,.$$
    The formula also applies for $n=\infty$ and $g(x)$ if \eqref{eq: explosion} holds, i.e. $\sum_{k=1}^\infty\frac{1}{F(k)}<\infty$.
\end{lemma}

We continue our proof of Theorem \ref{thm: birthMain} with another technical lemma.

\begin{lemma}\label{lemma: birth}
    Assume that (\ref{eq: explosion}) is fulfilled and that $F$ is strictly monotone. Then we have for any $t_0>0$ that
    \begin{equation*}
        \frac{g_k(t)}{g(t)}\xrightarrow{}1\quad\text{for }k\to\infty \text{ uniformly in }t\ge t_0\,.
    \end{equation*}
\end{lemma}

\begin{proof}
    Let $\epsilon>0$. For all $t_0>0$, there is $\delta>0$ such that $|g(t+s)/g(t)-1|<\epsilon$ for $0<s<\delta$ and all $t>t_0$, because the logarithmic derivative of $g$ is bounded (see Lemma \ref{lemma: gProperties}). Let $k$ be large enough such that $\P(\bar T(k)>\delta)<\epsilon$. Then we have by the convolution formula for $T(k)=T-\bar T(k)$
    \begin{align*}
        \left|\frac{g_k(t)}{g(t)}-1\right|&=\left|\frac{1}{g(t)}\int_t^\infty g(s)\bar g_k(s-t)ds-1\right|\\
        &\le\left|\frac{1}{g(t)}\int_t^{t+\delta} g(s)\bar g_k(s-t)ds-1\right|+\left|\frac{1}{g(t)}\int_{t+\delta}^\infty g(s)\bar g_k(s-t)ds\right|\\
        &\le 2\epsilon+\epsilon \frac{\max_{s\ge t+\delta}g(s)}{g(t)}
    \end{align*}
for all $t>0$. Properties 1.~and 2.~of Lemma \ref{lemma: gProperties} complete the proof.
\end{proof}

Note that obviously $g_k$ converges to $g$ globally uniformly for $t\ge0$, but Lemma \ref{lemma: birth} is a stronger assertion. We are now ready to finish the proof of Theorem \ref{thm: birthMain}.

\begin{proof}[Proof of Theorem \ref{thm: birthMain}]
    By the independence of $ T(k)$ and $\tau(k+1)$, we get:
    \begin{align*}
        &\P(\Xi(t)=x)=\P( T(x-1)<t,\,  T(x)\ge t)=\P(t-\tau(x)\le  T(x-1)<t)\\
        &=F(x)\int_0^\infty \P(t-u\le  T(x-1)<t)\,e^{-F(x)u}\,du\\
        &=F(x)\int_0^{a_x} \P(t-u\le  T(x-1)<t)\,e^{-F(x)u}\,du+F(x)\int_{a_x}^\infty \P(t-u\le  T(x-1)<t)\,e^{-F(x)u}\,du\,,
    \end{align*}
    where $a_x\coloneqq \frac{1}{\sqrt{F(x)}}$. Using the mean value theorem and the uniform asymptotic $g_x\sim g$ for $x\to\infty$ (see Lemma \ref{lemma: birth}), this implies for some intermediate value $\tilde t(u)\in(t-u, t)$
    \begin{align}\label{eq: meanValueTheorem}
        \P(t-u\le T(x)<t)=g_{x}(\tilde t(u))u\sim g(t)u\quad\text{for }u\to0,\, x\to\infty
    \end{align}
     Hence, the first integral is asymptotically
     \begin{align*}
         F(x)&\int_0^{a_x} \P(t-u\le  T(x-1)<t)\,e^{-F(x)u}\,du\sim g(t)\int_0^{a_x}u\,F(x)e^{-F(x)u}\,du\\
         &=-g(t)\left(a_x+\frac{1}{F(x)}\right)e^{-F(x)a_x}+\frac{g(t)}{F(x)}\sim\frac{g(t)}{F(x)}\quad\text{for }x\to\infty\,.
     \end{align*}
      The second integral is bounded by
     \begin{align*}
         F(x)\int_{a_x}^\infty \P(t-u\le  T(x-1)<t)\,e^{-F(x)u}\,du\le \int_{a_x}^\infty F(x)\,e^{-F(x)u}\,du= e^{-F(x)a_x}\,,
     \end{align*}
     which converges to zero faster than $\frac{1}{F(x)}$ for $x\to\infty$.
\end{proof}

\subsection{Discussion and extensions}

With Theorem \ref{thm: birthMain} the tail of $\Xi (t)$ conditioned on non-explosion $T>t$ is heavier the slower the feedback function $F$ increases. This may seem paradoxical at first and can be explained as follows: For fast increasing $F$, if conditioned on $T$ to be large, it is likely that the first summands of $T$ are large and the chain spends a long time in states with the smallest exit rates so that $\Xi(t)$ is rather small. This corresponds to the general principle in large deviation theory that the rare event $T>t$ (for $t\to\infty$) is realized in the most likely way. We get back to this in more detail in Section \ref{sec: tinfty}. On the other hand, the prefactor $-\frac{d}{dt}\log(\P(T>t))$ for small $t$ in Theorem \ref{thm: birthMain} is significantly smaller when $F$ increases slowly, as is visible in Figure \ref{figure: ParetoFaktor}.

Theorem \ref{thm: birthMain} also directly implies
\begin{equation}\label{eq: moments}
    \E[\Xi(t)^r\,|\, T>t]<\infty\quad\Leftrightarrow\quad \sum_{k=1}^\infty \frac{k^r}{F(k)}<\infty
\end{equation}
for any $r>0$. Let us now discuss some interesting examples.

\begin{example}\label{example: birth}
    \begin{enumerate}
        \item Let $F(k)=k^\beta$ for $\beta>1$. Then
        \begin{equation}\label{eq: loserTailSuperlin}
            \P(\Xi(t)> x\,|\,T>t)\sim\frac{-1}{\beta -1}\,\frac{d}{dt}\log(\P(T>t))\, x^{1-\beta}\quad\text{for }x\to\infty
        \end{equation}
        is a power-law distribution with exponent $1-\beta$, and
        $$\E[\Xi(t)^r\,|\, T>t]<\infty\quad\Leftrightarrow\quad r<\beta-1\,.$$
        
        \item Let $F(k)=e^{\beta k}$ for $\beta>0$. Then 
        $$\P(\Xi(t)> x\,|\,T>t)\sim-\frac{e}{\beta(e-1)}\,\frac{d}{dt}\log(\P(T>t))\, e^{-\beta (x+1)}\quad\text{for }x\to\infty$$
        has an exponential tail and $\E[\Xi(t)^r\,|\, T>t]<\infty$ for all $r>0$. Here and in the first example the tail is lighter the larger $\beta$ is.
        
        \item Let $F(k)=k(\log k)^\beta$ for $\beta>1$. Then
        \begin{equation}\label{eq: loserTailbeta>1}
            \P(\Xi(t)> x\,|\,T>t)\sim\frac{-1}{\beta-1}\,\frac{d}{dt}\log(\P(T>t))\, (\log x)^{1-\beta}\quad\text{for }x\to\infty
        \end{equation}
        decays logarithmically slowly and $\E[\Xi(t)^r\,|\, T>t]=\infty$ for all $r>0$.

    \end{enumerate}
\end{example}

\noindent Obviously we have
$$\lim_{t\to0}\P(\Xi(t)=\Xi(0)\,|\,T>t)=1$$
for all feedback functions, but remarkably the tail asymptotics for $\Xi(t)$ hold for any positive $t>0$, in particular for polynomially growing feedback it has a heavy-tailed distribution. This is consistent with vanishing prefactors in Theorem \ref{thm: birthMain}, i.e.
$$\lim_{t\to0}\frac{d}{dt}\log\P(T>t) =0\quad\text{and}\quad g(0)=0\,,$$
which is established in Lemma \ref{lemma: gProperties} and illustrated in Figure \ref{figure: ParetoFaktor}.
Since, $\Xi(t)=\infty$ is likely for large $t$, the prefactor $g(t)$ of the mass function vanishes for large $t$, while that of the tail conditioned on non-explosion $T>t$ converges to the minimum of $F$ (see \ref{eq: prelim} in Section \ref{sec: tinfty}).

By replacing the mean value theorem in (\ref{eq: meanValueTheorem}) by second order Taylor expansion,
$$\P(t-u< T(x)<t)=g_{x}(t)u+\frac12g_{x}'(\tilde t(u))u^2\,,$$
we can also determine the rate of convergence in Theorem \ref{thm: birthMain}:
\begin{align*}
    \P(\Xi(t)= x)&-\frac{g(t)}{F(x)}\sim\frac{g'(t)}{2}\int_0^\infty u^2\,F(x)\,e^{-F(x)dx}=\frac{g'(t)}{F(x)^2}\quad\text{for }x\to\infty
\end{align*}
Furthermore, we can also derive an asymptotic upper bound, which is even globally uniform.

\begin{corollary}\label{cor: uniformBound}
     Assume that (\ref{eq: explosion}) is fulfilled. Then
     $$\P(\Xi(t)=x)\prec \frac{g_x(t)}{F(x)}\quad\text{for }x\to\infty$$
     and
       $$\P(\Xi(t)> x\,|\,T>t)\prec \frac{g_x(t)}{\P(T>t)}\sum_{k=x+1}^\infty\frac{1}{F(k)}\quad\text{for }x\to\infty$$
     holds globally uniformly in $t\ge0$.
\end{corollary}

\begin{proof}
    Using Theorem \ref{thm: birthMain}, we only have to derive the uniform asymptotic bound for small $t\le t_0$. According to Lemma \ref{lemma: gProperties}, there is $t_0>0$ such that $g_x$ and $g$ is increasing on the interval $(0, t_0)$.  Then the proof of  Theorem \ref{thm: birthMain} applies analogously for $t\le t_0$ if we just replace (\ref{eq: meanValueTheorem}) by the (non-asymptotic) inequality
$$ \P(t-u< T(x)<t)=g_x(\tilde t(u))u\le g_x(t)u\,.$$
\end{proof}

As a consequence of 4. in Lemma \ref{lemma: gProperties}, $\lim_{t\to 0}\frac{g_k(t)}{g(t)}=\infty$ holds for all $k\ge1$ and, consequently,  the convergence in Lemma \ref{lemma: birth} is not uniform in $t\ge0$. For this reason, there is no uniform lower bound corresponding to Corollary \ref{cor: uniformBound}. In particular, the assumption $t_0>0$ is necessary in Theorem \ref{thm: birthMain}. This is already hinted at by the blue line in Figure \ref{figure: birthSimulation} (a).

Applying Corollary \ref{cor: uniformBound}, it is possible to randomize the observation time $t$ in Theorem \ref{thm: birthMain}, which turns out to be useful for the application in Section \ref{sec: applicatonPolya}.

\begin{corollary}\label{cor: birthMain}
     Assume that (\ref{eq: explosion}) is fulfilled. Let $S$ denote a positive random variable, which is independent of the process $\Xi$. Then as $x\to\infty$
      $$\P(\Xi(S)=x)\sim \frac{\E g(S)}{F(x)}\quad\text{and}\quad P(\Xi(S)> x|T>S)\sim \frac{\E g(S)}{\P (T>S)}\sum_{k=x+1}^\infty \frac{1}{F(k)}\,.$$ 
\end{corollary}
\begin{proof}
    Denote by $\P_S$ the law of $S$. Take $\epsilon>0$. Then:
    \begin{align*}
        \P(\Xi(S)=x)=\int_0^\infty \P(\Xi(s)=x)\,d\P_S(s)=\int_0^\epsilon \P(\Xi(s)=x)\,d\P_S(s)+\int_\epsilon^\infty \P(\Xi(s)=x)\,d\P_S(s)
    \end{align*}
Let us now look at the second integral and apply Theorem \ref{thm: birthMain}.
\begin{align*}
    \int_\epsilon^\infty \P(\Xi(s)=x)\,d\P_S(s)\sim  \int_\epsilon^\infty \frac{g(s)}{F(x)}\,d\P_S(s)\xrightarrow{\epsilon\to 0}\frac{\E g(S)}{F(x)}
\end{align*}
The first integral is asymptotically negligible due to Corollary \ref{cor: uniformBound}:
\begin{align*}
    \int_0^\epsilon \P(\Xi(s)=x)\,d\P_S(s)\prec  \int_0^\epsilon \frac{g_x(s)}{F(x)}\,d\P_S(s)\xrightarrow{\epsilon\to 0}0
\end{align*}
The second asymptotic follows via
\begin{align*}
     P(\Xi(S)> x|T>S)=\frac{\P(\Xi(S)> x,\, T>S)}{\P(T>S)}=\frac{\sum_{k=x+1}^\infty\P(\Xi(S)=k)}{\P(T>S)}\sim \frac{\E g(S)}{\P (T>S)}\sum_{k=x+1}^\infty \frac{1}{F(k)}
\end{align*}
for $x\to\infty$. Note that $\Xi(S)=k$ implies $T>S$.
\end{proof}

\subsection{Behaviour for $t\to\infty$\label{sec: tinfty}}

Figure \ref{figure: ParetoFaktor} also indicates that $-\frac{d}{dt}\log(\P(T>t))$ is increasing in $t$ and converges for $t\to\infty$. The latter observation can be rigorously confirmed for strictly monotone $F$, since we get from Lemma \ref{lemma: zhu} that
\begin{equation*}\label{eq: gasmypt}
    g(t)\sim F(\Xi(0))\left(\prod_{l=\Xi(0)+1}^\infty\frac{F(l)}{F(l)-F(\Xi(0))}\right)e^{-F(\Xi(0))t}\quad\text{for }t\to\infty
\end{equation*}
and by integration
\begin{align}\label{eq: Ttail}
    \P(T>t)\sim \left(\prod_{l=\Xi(0)+1}^\infty\frac{F(l)}{F(l)-F(\Xi(0))}\right)e^{-F(\Xi(0))t}\quad\text{for }t\to\infty\,.
\end{align}
Thus,
\begin{equation}\label{eq: prelim}
\lim_{t\to\infty}-\frac{d}{dt}\log(\P(T>t))=\lim_{t\to\infty}\frac{g(t)}{\P(T>t)}=F(\Xi(0))\,.    
\end{equation}
As a consequence, it stands to reason that $\lim_{t\to\infty}\P(\Xi(t)\in\cdot\,|\,T>t)$ defines a non-degenerated distribution. The following Proposition confirms this conjecture.

\begin{proposition}\label{prop: ttoinfty}
    Assume that (\ref{eq: explosion}) is fulfilled and that $F$ is strictly monotone. Then
    $$\P(\Xi(t)>x\,|\,T>t)\to 1-\prod_{k=x+1}^\infty\left(1-\frac{F(\Xi(0))}{F(k)}\right)\in (0,1)\quad\mbox{as }t\to\infty$$
    for all $x\ge\Xi(0)$. Moreover, we have for any $s>0$
    \begin{align*}
        \P(T-\tau(\Xi(0))>s\,|\, T>t)\prec  e^{-(F(\Xi(0)+1)-F(\Xi(0)))s}\quad\text{as }t\to\infty
    \end{align*}
   uniformly in $s$.
\end{proposition}

\noindent\textbf{Remarks.}
\begin{itemize}
    \item The so-called \textit{quasi-limiting distribution} $\nu_{\Xi (0)}^\infty :=\lim\limits_{t\to\infty} \P (\Xi (t)\in .\mid T>t)$ on $\{\Xi (0),\Xi (0)+1 ,\ldots\}$ has been studied in the context of \textit{Yaglom limits} and \textit{quasi-stationary distributions} for general Markov processes and in particular also for birth-death processes \cite{Cavender1978,Martinez1993}. In this limiting distribution, we can easily recover the tail $x\to\infty$ from Theorem \ref{thm: birthMain}
\begin{align*}
    \lim_{t\to\infty} \P(\Xi(t)>x\,|\,T>t)&=1-\prod_{k=x+1}^\infty\left(1-\frac{F(\Xi(0))}{F(k)}\right)=1-\exp\left\{\sum_{k=x+1}^\infty\log\left(1-\frac{F(\Xi(0))}{F(k)}\right)\right\}\\
    &\sim 1-\exp\left\{-F(\Xi(0))\sum_{k=x+1}^\infty\frac{1}{F(k)}\right\}\sim F(\Xi(0))\sum_{k=x+1}^\infty\frac{1}{F(k)}\quad\mbox{as }x\to\infty\ .
\end{align*}
Note that 
$$\lim_{x\to\infty} 1-\prod_{k=x+1}^\infty\left(1-\frac{F(\Xi(0))}{F(k)}\right)=0\,.$$
Hence, $\nu_{\Xi (0)}^\infty$ is a normalized probability distribution on $\{\Xi (0),\Xi (0)+1 ,\ldots\}$, in particular there is no probability mass in $x=\infty$. We give explicit examples for this distribution with power-law tails below.
    
    \item If a birth-death process is irreducible except for one absorbing state (usually the state $0$) whose hitting time has exponential moments, then quasi-limiting distributions are also quasi-stationary, i.e. do not change under the time evolution conditioned on non-absorption. Since we consider a pure birth process where all states are transient those results do not apply here. In particular, the quasi-limiting distributions $\nu_{\Xi (0)}^\infty$ depend on the initial condition and are not quasi-stationary. In fact, there are no quasi-stationary distributions for our process. Due to transience of all states, the only possible candidates would be $\delta_{\Xi (0)}$, but with Proposition \ref{prop: ttoinfty} $\nu_{\Xi (0)}^\infty (\Xi (0)) <1$ for all feedback functions.
    
    \item But how does the event of non-explosion at some late time typically occur? According to Proposition \ref{prop: ttoinfty}, the time $\tau(\Xi(0))$ is of the same order as the explosion time $T$. Moreover, since the tail of $T$ decays exponentially, $T-t$ has an exponential tail as well on the event $\{T>t\}$. Hence, conditioned on  $\{T>t\}$ for large $t$, the first jump of the process happens close to $t$ up to a light tailed distribution (asymptotically independent of $t$), i.e.
$$\P(|\tau(\Xi(0))-t|>s\,|\,T>t)\prec  e^{-ms}\quad\text{as }t\to\infty$$
uniformly in $s$, where $m\coloneqq \min\{F(\Xi(0)), F(\Xi(0)+1)-F(\Xi(0))\}>0$. In particular this implies $\P (\Xi (\alpha t)=\Xi (0)|T>t )\to 1$ as $t\to\infty$ for all $\alpha <1$. Nevertheless, there might be further jumps between the first one at $\tau(\Xi(0))$ and the observation time $t$ with non-vanishing probability. Note that for $ F(\Xi(0)+1)\approx F(\Xi(0))$ there can be several long sojourn times when $T>t$.\\
\end{itemize}

\begin{proof}[Proof of Proposition \ref{prop: ttoinfty}]
    Due to the strict monotonicity of $F$ and Lemma \ref{lemma: zhu}, it is easy to see that 
    \begin{equation*}\label{eq: gkasmpt}
        g_k(t)\sim F(\Xi(0))\left(\prod_{l=\Xi(0)+1}^k\frac{F(l)}{F(l)-F(\Xi(0))}\right)e^{-F(\Xi(0))t}\quad\text{for }t\to\infty
    \end{equation*}
    for $k\ge\Xi(0)$ including $k=\infty$ with $g_\infty=g$. Then the first claim follows via de l'Hospital's Theorem:
    \begin{align*}
        \P(\Xi(t)>x\,|\,T>t)&=1-\P(\Xi(t)\le x\,|\,T>t)=1-\frac{\P(\Xi(t)\le x)}{\P(T>t)}=1-\frac{\P( T(x)>t)}{\P(T>t)}\sim 1-\frac{g_x (t)}{g(t)}\\
        &\xrightarrow{t\to\infty}1-\frac{ F(\Xi(0))\left(\prod_{k=\Xi(0)+1}^x\frac{F(k)}{F(k)-F(\Xi(0))}\right)}{ F(\Xi(0))\left(\prod_{k=\Xi(0)+1}^\infty\frac{F(k)}{F(k)-F(\Xi(0))}\right)}=1-\prod_{k=x+1}^\infty\left(1-\frac{F(\Xi(0))}{F(k)}\right)
    \end{align*}
    As in (\ref{eq: Ttail}), we can characterize the tail distribution of $T$ and $\bar T(\Xi(0))$. This allows the following asymptotic calculation:
    \begin{align*}
         \P&(T-\tau(\Xi(0))>s\,|\, T>t)= \frac{\P(\bar T(\Xi(0))>s,\, T>t)}{\P(T>t)}= \frac{ \P(\bar T(\Xi(0))>\max\{t-\tau(\Xi(0)), s\})}{\P(T>t)}\\
         &=\frac{1}{\P(T>t)}\int_0^\infty \P(\bar T(\Xi(0))>\max\{t-u, s\})F(\Xi(0))e^{-F(\Xi(0))u} du\\
         &=\frac{F(\Xi(0))}{\P(T>t)}\left(\int_0^{t-s} \P(\bar T(\Xi(0))>t-u)e^{-F(\Xi(0))u}du+\int_{t-s}^\infty \P(\bar T(\Xi(0))>s)e^{-F(\Xi(0))u}du\right)\\
         &\le \frac{F(\Xi(0))}{\P(T>t)}\left(const. \int_0^{t-s} e^{-F(\Xi(0)+1)(t-u)}e^{-F(\Xi(0))u}du+const. e^{-F(\Xi(0)+1)s}e^{-F(\Xi(0))(t-s)}\right)\\
         &\sim \frac{const.}{e^{-F(\Xi(0))t}}\left(e^{-F(\Xi(0)+1)t}\left(1-e^{(F(\Xi(0)+1)-F(\Xi(0)))(t-s)}\right)+ e^{-F(\Xi(0)+1)s}e^{-F(\Xi(0))(t-s)}\right)\\
         &\sim const.\left( e^{-(F(\Xi(0)+1)-F(\Xi(0)))s} +e^{-(F(\Xi(0)+1)-F(\Xi(0)))s}\right)\quad\text{for }t\to\infty
    \end{align*}
    Note that $const.$ is independent of $s$ and $t$.
\end{proof}

\begin{example}
    For $F(k)=k^\beta$ with $\beta\in\{2, 3,\ldots\}$, we have the general formula
    $$\prod_{k=x+1}^\infty\left(1-\frac{F(\Xi(0))}{F(k)}\right)=\Gamma(x+1)^\beta\prod_{k=0}^{\beta-1}\frac{1}{\Gamma(x+1-\Xi(0)e^{\frac{2\pi i k}{\beta}})}\,,$$
    where $\Gamma$ denotes the Gamma function. Using $\Gamma(k)=(k-1)!$, we get for $\beta=2$:
    $$\lim_{t\to\infty}\P(\Xi(t)>x\,|\,T>t)=\frac{\Gamma (x+1)}{\Gamma (x+1+\Xi (0))}=\prod_{k=1}^{\Xi (0)}\frac{1}{x+k}$$
    In particular, for $x=\Xi (0)$ we get for the probability that the chain has moved at all
    $$
    \lim_{t\to\infty}\P(\Xi(t)>\Xi (0)\,|\,T>t) =\frac{\Xi (0)!}{(2\Xi (0))!}
    $$
    which vanishes for large initial condition $\Xi (0)\to\infty$. On the other hand, for 
    $\Xi (0)=1$ we have
    $$
    \lim_{t\to\infty}\P(\Xi(t)>x\,|\,T>t) =\frac{1}{x+1} =\frac12 \quad\mbox{for }x=\Xi (0)\ .
    $$
    For $\beta\ne2$, $\beta>1$ and $\Xi (0)=1$ considering the increments yields that the approximation
     $$\lim_{t\to\infty}\P(\Xi(t)>x\,|\,T>t)\approx\frac{1}{(\beta-1)(x+1)^{\beta-1}}$$
     is accurate even for small $x$, but not precise. The good precision of this approximation for large $t$ is visible in Figure \ref{figure: birthSimulation}. 
\end{example}

The discussion in this section does not crucially rely on strict monotonicity of $F$ and we assumed it for simplicity of the presentation and since it is quite natural in many applications. The results could be extended directly to feedback functions with a unique minimum, where the non-exploding chain will spend most of its time.

\section{Application to generalized Pólya urns}\label{sec: applicatonPolya}

The generalized Pólya urn model with $A\in\{2, 3, \ldots\}$ agents is formally defined as a discrete-time Markov process $(X(n))_{n\in\N_0}=((X_1(n),\ldots,X_A(n)))_{n\in\N_0}$ on the state space $\N^A$ with transition probabilities given by 
$$\P\left(X(n+1)-X(n)=e^{(i)}\,|\, X(n)\right)=\frac{F_i(X_i(n))}{F_1(X_1(n))+\ldots F_A(X_A(n))}$$
and initial configuration $X(0)\in\N^A$, where $F_i\colon \N\to(0, \infty)$ is called \textit{feedback function} of agent $i\in[A]$. Here, $e^{(i)}\coloneqq (\delta_{i, j})_{j\in[A]}$ denotes the $i$-th unit vector, and in each time-step one unit enters the system and is won by agent $i$ with probability proportional to the feedback function $F_i(X_i(n))$. An extensive analysis of this model can be found in \cite{wir} and the references therein.

The so-called \textit{exponential embedding} (see e.g. \cite{Oliveira} an references therein) provides a useful alternative construction of this process. For that, we take $A$ independent birth processes $(\Xi_i(t))_{t\ge0}$ with rate function $F_i$ and initial value $\Xi_i(0)=X_i(0)$ as defined in the introduction. We denote the sojourn times of $\Xi_i$ by $\tau_i(k)$, which are independent and exponentially distributed with parameter $F_i(k)$. Moreover, denote by $(t_n)_n$ the random sequence of jump times of the combined process $t\mapsto\Xi (t)=\big(\Xi_i (t):i\in [A]\big)$, i.e.
$$t_{n+1}\coloneqq\inf\{t>t_n\colon \Xi(t)\ne\Xi(t_n)\}\quad\text{with }t_0\coloneqq0\,.$$
Then Rubin's Theorem states that the discrete process $(X(n))_{n\in\N_0}\coloneqq(\Xi(t_n))_{n\in\N_0}$ describes a generalized Pólya urn as defined above. In the rest of the paper we will always use this construction and $\P$ denotes the law of $A$ independent birth processes with rate functions $F_1 ,\ldots ,F_A$.

An important event is the occurrence of \textit{strong monopoly}, where one agent wins all but finitely many steps of the process:
$$sMon_i\coloneqq\{\omega\in\Omega\colon \exists N\ge0\forall n\ge N\,,\, j\ne i\colon X_j(n)=X_j(N)\}$$
Against the background of the exponential embedding, it is easy to see (see \cite{wir} and references therein for details) that strong monopoly occurs with probability one if and only if at least one feedback function fulfills (\ref{eq: explosion}). In the following we will always assume that this is the case. Moreover, $\P(sMon_i)>0$ if and only if $F_i$ satisfies (\ref{eq: explosion}).

\subsection{Asymptotic size of losing agents}

Conditioned on agent $i$ losing, i.e.\ not to be the monopolist, $X_i(\infty)\coloneqq\lim_{n\to\infty}X_i(n)$ is a finite random variable, whose tail distribution can be computed with Corollary \ref{cor: birthMain}. Denote by 
$$T_i\coloneqq \sup\{t>0\colon \Xi(t)<\infty\}\in(0, \infty]$$
the (random) explosion time of $\Xi_i$ and if $F_i$ satisfies (\ref{eq: explosion}), denote by 
$$g_i(t)\coloneqq \frac{d}{dt}\P(T_i\le t)$$
the density of $T_i$. This leads to our second main result, which generalizes previous results in \cite{Cotar, Oliveira, Zhu}.

\begin{theorem}\label{thm: loserSuperlin}
    Assume that at least two agents fulfill (\ref{eq: explosion}) and let agent $i\in[A]$ be one of them. Then as $x\to\infty$
    $$\P\left(X_i(\infty)=x\right)\sim \frac{\E g_i\left(\min_{j\ne i}T_j\right)}{F_i(x)}\quad\text{and}\quad P(X_i(\infty) > x|sMon_i^c )\sim \frac{\E g_i\left(\min_{j\ne i}T_j\right)}{\P (T_i >\min_{j\neq i}T_j)}\sum_{k=x+1}^\infty \frac{1}{F_i (k)}\ .$$
\end{theorem}

\begin{proof}
    This is an immediate consequence of Corollary \ref{cor: birthMain} since $S\coloneqq\min_{j\ne i}T_j$ is independent of $(\Xi_i (t))_{t\geq 0}$ and
    $$\{X_i(\infty)=x\}=\{\Xi_i(S)=x\}\quad\mbox{as well as}\quad sMon_i^c =\{ T_i >S\}\ .$$
    Note that $\E g_i\left(\min_{j\ne i}T_j\right)\in(0, \infty)$ as $g_i$ is bounded (Lemma \ref{lemma: gProperties}) and at least one $T_j$ is finite.
\end{proof}

\begin{figure}
  \centering
  \subfloat[$F_1(k)=k^2$, $F_2(k)=k^3$]{\includegraphics[width=0.5\linewidth]{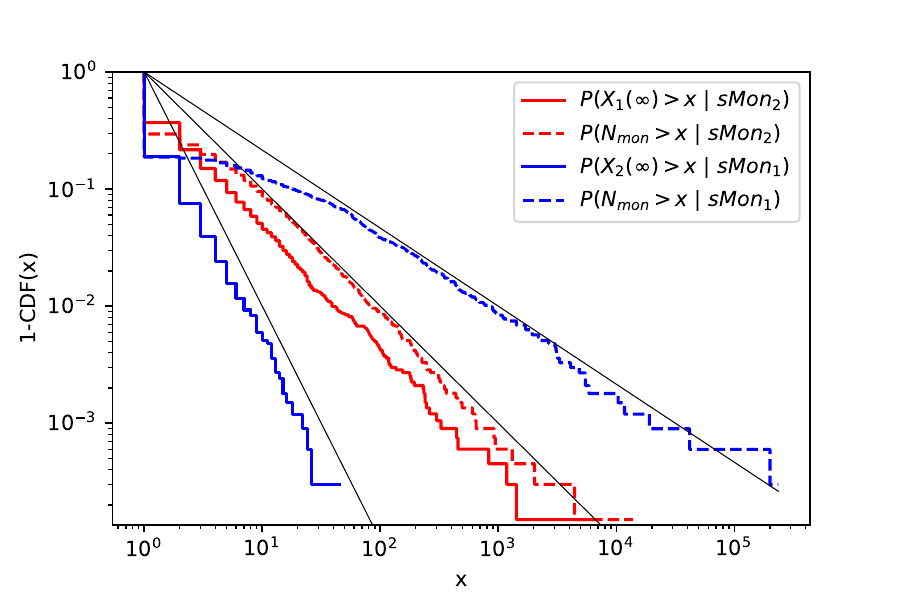}}
  \subfloat[$F_1(k)=k^2$, $F_2(k)=e^k$]{\includegraphics[width=0.5\linewidth]{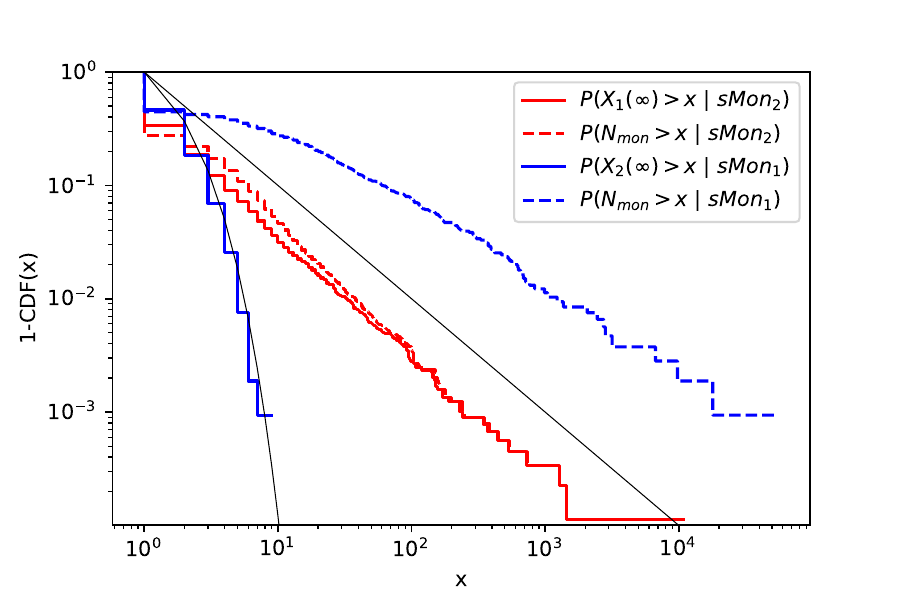}}
  \caption{Empirical distribution of $X_i(\infty)$ and $N_{mon}$ (\ref{eq: MonTimeDef}) for $A=2$ and $X_1(0)=X_2(0)=1$. The black lines show the tail decay predicted by Theorem \ref{thm: loserSuperlin} and Corollary \ref{cor: MonTimeSuperlin}. Note that these predicted tails do only claim exactness  up to constants, i.e. up to a parallel shift in the plot. 10,000 simulations were executed each.}
  \label{figure: LoserSimulation}
\end{figure}

  Figure \ref{figure: LoserSimulation} illustrates Theorem \ref{thm: loserSuperlin} for power-law feedback. Remarkably, the tail of $X_i(\infty)$ conditioned on $sMon_i^c$ does not depend on the feedback functions of other agents up to a constant prefactor. In particular, $X_i(\infty)$ has a power-law tail even if there is another agent with exponential feedback. For the existence of moments of $X_i(\infty)$ conditioned on $sMon_i^c$, criterion (\ref{eq: moments}) holds analogously. As explained in Example \ref{example: birth}, $\P(X_i(\infty)>x\,|\,sMon_i^c)$ has a power-law tail for polynomially increasing $F_i$, whereas it has exponential tail for exponentially increasing $F_i$. For $F_i(k)=k(\log k)^\beta$ with $\beta>1$, the tail decay is even logarithmically slow.\\

\noindent\textbf{Remarks.}
\begin{itemize}
    \item Theorem \ref{thm: loserSuperlin} implies that the tail of $X_i(\infty)$ conditioned on $sMon_i^c$ is heavier the weaker the feedback is. This means that among agents with feedback satisfying \eqref{eq: explosion}, losers with slowly increasing $F_i$ are more likely to win in many steps than other losers with stronger feedback. In analogy to the same phenomenon for the single birth process, if an agent with feedback satisfying \eqref{eq: explosion} wins in many steps, then they are likely to be the monopolist. On the other hand, if this agent is not the monopolist, then they are likely to have won only a few steps.

\item Let us now consider a system, where at least one agent fulfills (\ref{eq: explosion}), but not agent $i$. Then the process exhibits strong monopoly, but $\P(sMon_i)=0$ and hence $\P(X_i(\infty)<\infty)=1$ and agent $i$ loses. Since such agents are not affected by conditioning on $sMon_i^c$, via canonical coupling it is obvious that $X_i(\infty)$ is stochastically larger the faster $F_i$ increases. So among agents with feedback that does not satisfy \eqref{eq: explosion} those with stronger feedback are likely to win more steps.

\item Summarizing both cases, the tail of $X_i(\infty)$ for a losing agent (possibly conditioning on $sMon_i^c$) is heavier for feedback functions closer to the transition given by \eqref{eq: explosion}. For a better understanding of this tail-maximising phenomenon, a detailed discussion of the tail behaviour close to the transition (\ref{eq: explosion}) is provided in Section \ref{sec: LoserSublin}.
\end{itemize}

In general, the prefactor $\E g_i\left(\min_{j\ne i}T_j\right)$ in Theorem \ref{thm: loserSuperlin} is decreasing in $A$, i.e. the probability that a loser wins many steps is smaller in larger systems. Moreover, this prefactor vanishes for $A\to\infty$, if infinitely mans agents satisfy (\ref{eq: explosion}). This complies with the idea that in large systems a typical agent does not even win a single step (see Appendix \ref{sec: manyAgents}). Surprisingly, large systems behave totally different if only one agent satisfies (\ref{eq: explosion}), such that this agent is almost surely the monopolist. In this case,  the distribution of $X_i(\infty)$, is independent of the number of other agents as a direct consequence of the exponential embedding. Hence, in a world with super-linear  (polynomial) feedback and many agents, there will be a few agents with high wealth, but the vast majority will not win anything, whereas in a world with one dominant agent (the super-linear one) all sub-linear losers have good chances to win some steps.

\subsection{Correlations among losing agents}\label{subsec: loserCor}

Next, we have a closer look at the dependencies of several losing agents. Note that the correlation of $X_i(\infty)$ and $X_j(\infty)$ on the conditioned space $sMon_i^c\cap sMon_j^c$ does in general not exist due to infinite second moments, e.g. for $F_i(k)=k^2$.

\begin{theorem}\label{thm: LoserCorrelation}
     Let $a\in[A-1]$. Assume that at least $a+1$ agents satisfy (\ref{eq: explosion}) and that agents $1,\ldots, a$ are among them. Set $S\coloneqq \min\{T_i\colon i=a+1,\ldots, A\}$ and $S_i\coloneqq \min\{T_j\colon j\in[A]\backslash\{i\}\}$. Then:
    \begin{align*}
          \P&\left(X_1(\infty)= x_1,\ldots, X_a(\infty)= x_a \right)\sim\frac{\E\left[\prod_{i=1}^a g_i(S)\right]}{\prod_{i=1}^a\E g_i(S_i)}\cdot\prod_{i=1}^a\P\left(X_i(\infty)= x_i\right)\quad\text{for }x_1,\ldots, x_a\to\infty
     \end{align*}
\end{theorem}

\begin{proof}
Denote by $\P_S$ the law of $S$. Using the independence property of the exponential embedding, we get:
    \begin{align*}
         \P&\left(X_1(\infty)= x_1,\ldots, X_a(\infty)= x_a \right)=\P\left(\Xi_1(S)= x_1,\ldots, \Xi_a(S)= x_a \right)\\
         &=\int_0^\infty \P\left(\Xi_1(s)= x_1,\ldots, \Xi_a(s)= x_a \right)\,d\P_S(s)=\int_0^\infty \prod_{i=1}^a\P\left(\Xi_i(s)= x_i\right)\,d\P_S(s)\\
         &\sim\int_0^\infty \prod_{i=1}^a\frac{g_i(s)}{F_i(x_i)}\,d\P_S(s)=\E\left[\prod_{i=1}^a g_i(S)\right]\prod_{i=1}^a\frac{1}{F_i(x_i)}\quad\text{for }x_1,\ldots, x_a\to\infty
    \end{align*}
    In the last line, we apply Theorem \ref{thm: birthMain} and Corollary \ref{cor: uniformBound} in the same manner as in the proof of Corollary \ref{cor: birthMain}. The claim follows then via Theorem \ref{thm: loserSuperlin}.
\end{proof}

In particular, for polynomial feedback, $(X_1(\infty), X_2(\infty))$ has a two-dimensional heavy tailed distribution on $ sMon_1^c\cap sMon_2^c$ in the sense of \cite{konstantinides}.  Denote by
$$ \P_{A, a}(\cdot)\coloneqq\frac{\P((\cdot)\cap sMon_1^c\cap\ldots sMon_a^c)}{\P(sMon_1^c\cap\ldots sMon_a^c)}$$
the conditioned probability measure on $sMon_1^c\cap\ldots \cap sMon_a^c$. Then we can rephrase Theorem \ref{thm: LoserCorrelation} as
\begin{align}\label{eq: CorAsymp}
           \P_{A, a}\left(X_1(\infty)> x_1,\ldots, X_a(\infty)> x_a  \right)\sim c(A, a)\cdot\prod_{i=1}^a \P_{A, a}\left(X_i(\infty)> x_i\right)\quad\text{for }x_1,\ldots, x_a\to\infty\,,
     \end{align}
where 
     \begin{equation}\label{eq: cFormula}
         c(A, a)\coloneqq\frac{\E\left[\prod_{i=1}^a g_i (S)\right]}{\prod_{i=1}^a\E\left[\left(\prod_{j\in[a]\backslash\{i\}}G_j(S)\right)g_i (S)\right]}\cdot\P\left(sMon_1^c\cap\ldots\cap sMon_a^c \right)^{a-1}\,.
     \end{equation}
    and $G_i(s)\coloneqq\P(T_i>s)$, since for $x_i\to\infty$
    \begin{align}\label{eq: tildeP}
        & \P_{A, a}\left(X_i(\infty)= x_i\right)=\frac{\P(\Xi_i(S)=x_i,\,\forall j\in[a]\colon T_j>S)}{\P(sMon_1^c\cap\ldots sMon_a^c)}\\
        &=\frac{1}{\P(sMon_1^c\cap\ldots sMon_a^c)}\int_0^\infty \P(\Xi_i(s)=x_i,\,\forall j\in[a]\colon T_j>s)\,d\P_S(s)\nonumber\\
         &=\frac{1}{\P(sMon_1^c\cap\ldots sMon_a^c)}\int_0^\infty \P(\Xi_i(s)=x_i)\prod_{j\in[a]\backslash\{i\}}G_j(s)\,d\P_S(s)\nonumber\\
         &\sim\frac{1}{\P(sMon_1^c\cap\ldots sMon_a^c)}\int_0^\infty \frac{g_i(s)}{F_i(x_i)}\prod_{j\in[a]\backslash\{i\}}G_j(s)\,d\P_S(s)\quad\text{for }x\to\infty\nonumber\\
         &=\frac{\E\left[\left(\prod_{j\in[a]\backslash\{i\}}G_j(S)\right)g_i (S)\right]}{\P(sMon_1^c\cap\ldots sMon_a^c)}\cdot\frac{1}{F_i(x_i)}\sim\frac{\E\left[\left(\prod_{j\in[a]\backslash\{i\}}G_j(S)\right)g_i (S)\right]}{\P(sMon_1^c\cap\ldots sMon_a^c)\E g_i(S_i)}\cdot\P(X_i(\infty)=x_i)\nonumber\\
         &=\frac{\E\left[\left(\prod_{j\in[a]\backslash\{i\}}G_j(S)\right)g_i (S)\right]}{\P(sMon_1^c\cap\ldots sMon_a^c)\E g_i(S_i)}\cdot\P(sMon_i^c)\cdot\P(X_i(\infty)=x_i\,|\,sMon_i^c)\nonumber
    \end{align}
i.e. the tail of $X_i(\infty)$ conditioned on $sMon_i^c$ is equal to its tail conditioned on $sMon_1^c\cap\ldots\cap sMon_a^c$ up to a constant. This constant is necessarily one in a fully symmetric situation, since the distribution of $X_i(\infty)$ on $sMon_i^c$ is then independent of who the monopolist is. In a non-symmetric situation, the information who won contains information on the explosion time $S$, which affects the distribution of $X_i(\infty)$. E.g., if the an agent with strong feedback is the monopolist, then $S$ is likely to be small, such that $X_i(\infty)$ is rather small, too. Hence, when $F_1=\ldots=F_A$ and $X_1(0)=\ldots X_A(0)$, then we can simply write
$$c(A, a)=\frac{\E\left[ g(S)^a\right]}{\left(\E g(S)\right)^a}\cdot\frac{\prod_{i=1}^a\P(sMon_i^c)}{\P(sMon_1^c\cap\ldots sMon_a^c)}=\frac{\E\left[ g(S)^a\right]}{\left(\E g(S)\right)^a}\cdot\frac{(A-1)^a}{A^{a-1}(A-a)}\,,$$
where $g=g_i$ and $G=G_i$. The first factor is larger than one due to Jensen's inequality and the second factor is larger than one due to
\begin{equation}\label{eq: bernoulli}
    \frac{(A-1)^a}{A^{a-1}(A-a)}>1\quad\Leftrightarrow\quad \left(1-\frac 1 A\right)^a> 1-\frac a A
\end{equation}
and Bernoulli's inequality. Hence, we have $c(A, a)>1$ in the symmetric case.\\

In general, the dependence between several losers can be considered as weak since the tail weight of a loser does not change when the result of other losers is known. The remaining tail dependence is captured by the constant $c(A, a)$, which we shortly discuss now. Heuristically, the event that a loser wins in many steps can occur in two ways. First, it is possible that the underlying birth process of this agent increases fast. This effect is independent between several losers. Second, it is possible that the explosion time of the winner is relatively late, such that all losers are more likely to win in many steps. This second effect accounts for a slightly positive correlation of losers, corresponding to $c(A, a)>1$. The following proposition formally substantiates this idea also for non-diverging $x_i$.

\begin{proposition}\label{prop: c}
     Assume $A\ge3$, $X_1(0)=\ldots=X_A(0)$ and $F_1=\ldots=F_A=F$, where $F$ satisfies (\ref{eq: explosion}). Then we have $c(A, a)>1$ and
    \begin{align*}
         \P_{A, a}\left(X_1(\infty)>x,\ldots,\,X_a(\infty)>x \right)>\prod_{i=1}^a \P_{A, a}\left(X_i(\infty)>x\right)
    \end{align*}
     for any $x\ge X_1(0)$ and $a\in[A-1]$.
\end{proposition}

\begin{proof}
$c(A, a)>1$ has been derived above. Due to the assumed symmetry, we can define $h(s)\coloneqq\P(\Xi_i(s)>x)$. Denote by $\P_S$ the law of  $S$ (defined in Theorem \ref{thm: LoserCorrelation}). Using Jensen's inequality, we get:
\begin{align*}
     \P_{A, a}&\left(X_1(\infty)>x,\,\ldots,\, X_a(\infty)>x\right)=\frac{\P\left(\Xi_1(S)>x, \ldots, \Xi_a(S)>x \right)}{\P( sMon_1^c\cap\ldots\cap sMon_a^c)}\\
    &=\frac{1}{\P( sMon_1^c\cap\ldots\cap sMon_a^c)}\int_0^\infty\P\left(\Xi_1(s)>x,  \ldots, \Xi_a(s)>x\right)\,d\P_S(s)\\
    &=\frac{A}{A-a}\int_0^\infty h(s)^a\,d\P_S(s)>\frac{A}{A-a}\left(\int_0^\infty h(s)\,d\P_S(s)\right)^a\\
    &=\frac{A}{A-a}\P(\Xi_1(S)>x)^a>\frac{A}{A-a}\P(X_1(\infty)>x)^a\\
    &=\frac{A}{A-a}\left(\P(X_1(\infty)>x,\, sMon_1^c,\ldots, sMon_a^c)\cdot\frac{A-1}{A-a}\right)^a\\
    &=\frac{A(A-1)^a}{(A-a)^{a+1}}\prod_{i=1}^a\left[ \P_{A, a}\left(X_i(\infty)>x\right)\,\P( sMon_1^c,\ldots, sMon_a^c)\right]\\
    &=\frac{A(A-1)^a}{(A-a)^{a+1}}\cdot \left(\frac{A-a}{A}\right)^a\prod_{i=1}^a \P_{A, a}\left(X_i(\infty)>x\right)>\prod_{i=1}^a \P_{A, a}\left(X_i(\infty)>x\right)
\end{align*}
The last step follows again via Bernoulli's inequality. Throughout the calculation, we exploited the supposed symmetry of all agents. 
\end{proof}

Appendix \ref{sec: aDependence} examines the dependence of $c(A, a)$ on the system size $A$. In Appendix \ref{sec: manyAgents} (in particular Theorem \ref{thm: limA}) we will first see that for large $A$, a randomly chosen agent does not win a single step with high probability. Recall that $c(A, a)$ only measures the tail dependence of agents, i.e. we basically only consider agents who won many steps. Surprisingly, $c(A, a)$ turns out to be increasing with $A$, i.e. the dependence between losers in the sense captured by $c(A, a)$ is stronger in large systems. This seems to be opposed to the linear case, where we observe asymptotic independence for $A\to \infty$ (Theorem \ref{thm: asymptoticExp}), but we also have a trivial asymptotic independence in the super-linear case (Corollary \ref{cor: trivialIndependence}) since typical losers do not win anything. Moreover, $c(A, a)$ is increasing in $a$ since the information that several losers won many steps is a stronger hint for a late explosion time of the winner, such that the considered agent is more likely to win many steps, too.

\subsection{Further consequences of Theorem \ref{thm: loserSuperlin} and Theorem \ref{thm: LoserCorrelation}}

Theorem \ref{thm: loserSuperlin} and Theorem \ref{thm: LoserCorrelation} are useful to determine the tail behaviour of further random variables arsing in the context of our generalized Pólya urn model. This subsection presents some examples.

\begin{corollary}\label{cor: TailCorAsymp}
    In the situation of Theorem \ref{thm: LoserCorrelation}, the following asymptotics hold for $x\to\infty$.
    \begin{enumerate}
        \item \begin{align*}
             \P_{A, a}&(\min(X_1(\infty),\ldots, X_a(\infty))>x)\sim \frac{\E\left[\prod_{i=1}^ag_i(S)\right]}{\P(sMon_1^c\cap\ldots\cap sMon_a^c)}\cdot\prod_{i=1}^a\left(\sum_{k=x+1}^\infty \frac{1}{F_i(k)}\right)
        \end{align*}
        \item  \begin{align*}
             \P_{A, a}&(\max(X_1(\infty),\ldots, X_a(\infty))>x)\sim \sum_{i=1}^a\left(\frac{\E\left[\left(\prod_{j\ne i}G_j (S)\right)g_i(S)\right]}{\P(sMon_1^c\cap\ldots\cap sMon_a^c )}\sum_{k=x+1}^\infty \frac{1}{F_i(k)}\right)
        \end{align*}
        \item If $F_1,\ldots, F_a$ are regular varying, then:
         \begin{align*}
             \P_{A, a}(X_1(\infty)+\ldots+X_a(\infty)>x)\sim \sum_{i=1}^a\left(\frac{\E\left[\left(\prod_{j\ne i}G_j(S)\right)g_i(S)\right]}{\P(sMon_1^c\cap\ldots\cap sMon_a^c )}\sum_{k=x+1}^\infty \frac{1}{F_i(k)}\right)
        \end{align*}
    \end{enumerate}
\end{corollary}

\begin{proof}
    1. We rephrase  Theorem \ref{thm: LoserCorrelation} again as
    $$\P_{A, a}(X_1(\infty)>x_1,\ldots, X_a(\infty)>x_a)\sim \frac{\E\left[\prod_{i=1}^ag_i(S)\right]}{\P(sMon_1^c\cap\ldots\cap sMon_a^c)} \cdot\prod_{i=1}^a\left(\sum_{k=x_i+1}^\infty \frac{1}{F_i(k)}\right)\,.$$
    Then the claim follows from
    \begin{align*}
         \P_{A, a}(\min(X_1(\infty),\ldots, X_a(\infty))>x)= \P_{A, a}\left(X_1(\infty)>x,\ldots,\,X_a(\infty)>x\right)\,.
    \end{align*}
    
    2. This follows from (\ref{eq: tildeP}) and the inclusion-exclusion formula via 
    \begin{align*}
         \P_{A, a}(\max(X_1(\infty),\ldots, X_a(\infty))>x)= \P_{A, a}\left(\bigcup_{i=1}^a\{X_i(\infty)>x\}\right)\sim \sum_{i=1}^a \P_{A, a}(X_i(\infty)>x)\,.
    \end{align*}
    
    3.  Karamata's Theorem implies that $x\mapsto\sum_{k=x+1}^\infty \frac{1}{F_i(k)}$ are also regularly varying for $i=1,\ldots,\,a$. Then the claim follows from (\ref{eq: tildeP}) and Lemma \ref{lemma: paretoSumTail} applied on the conditioned probability space $sMon_1^c\cap\ldots\cap sMon_a^c$.
\end{proof}

\begin{lemma}\label{lemma: paretoSumTail}
    Let $X_1, \ldots, X_k$ be random variables, such that $x\mapsto\P(X_i>x)$ is regular varying for all $i=1,\ldots,k$. Moreover, assume that the $X_i$ are at most weakly dependent in the sense that
    $$\forall i, j\in[k]\colon \lim_{x\to\infty}\frac{\P(X_i>x)}{\P\left(X_i>x,\,X_j>x\right)}=\infty\,.$$
    Then:
    $$\P(X_1+\ldots+X_k>x)\sim\P(X_1>x)+\ldots +\P(X_k>x)\quad\text{for }x\to\infty$$    
\end{lemma}

\begin{proof}
We only prove the case $k=2$ since general $k$ follows directly by induction. Then we have the following lower bound:
\begin{align*}
    \P(X_1+X_2>x)\ge \P(X_1>x)+\P(X_2>x)-\P(X_1>x,\,X_2>x)\sim\P(X_1>x)+\P(X_2>x)
\end{align*}
For an upper bound, take $0<\delta<\frac12$. Then
\begin{align*}
    \P(X_1+X_2>x)&\le \P(X_1>(1-\delta)x)+\P(X_2>(1-\delta)x)+\P(X_1>\delta x,\,X_2>\delta x)\\
    &\sim (1-\delta)^{\alpha_1}\P(X_1>x)+(1-\delta)^{\alpha_2}\P(X_2>x)\,,
\end{align*}
where $\alpha_i$ is the index of regular variation of $x\mapsto\P(X_i>x)$. $\delta\to0$ implies the claim.
\end{proof}

Note that in Corollary \ref{cor: TailCorAsymp} the maximum and sum have the same tail. This is a characterizing property of sub-exponential random variables and, hence, the assumption of regular varying feedback cannot be omitted in 3. of Corollary \ref{cor: TailCorAsymp}. 

Again, the tail decay of the sum, minimum and maximum of $a$ losers only depends on $F_1,\ldots, F_a$, but not on $F_{a+1},\ldots, F_A$ (up to the constant prefactor). Corollary \ref{cor: TailCorAsymp} also contains detailed information about the sum, minimum and maximum of all losers. To see that, denote by $W\in[A]$ the random winner of the process and write
$$\P\left(\sum_{j\ne W}X_j(\infty)>x\right)=\sum_{i=1}^A\P(sMon_i)\P\left(\sum_{j\ne i}X_j(\infty)>x\,\Big|\, \bigcap_{j\ne i}sMon_j^c\right)$$
and analogously for the maximum and minimum. Let us finally discuss this in the polynomial case.

\begin{example}
    Let $F_i(k)=\alpha_i k^{\beta_i}$ for $\alpha_i>0$ and $1<\beta_1\le\ldots\le\beta_A$. Then
    $$\P\left(\sum_{i\ne W}X_i(\infty)>x\right)\sim\P\left(\max_{i\ne W}X_i(\infty)>x\right)\sim const. x^{1-\beta_1} \quad\text{for }x\to\infty$$
    and 
    $$\P\left(\min_{i\ne W}X_i(\infty)>x\right)\sim const. x^{A-1-\beta_1-\ldots-\beta_{A-1}} \quad\text{for }x\to\infty\,.$$
    In both cases, the tail decay is independent of $\beta_A$ up to a constant. Moreover, if $\beta_1<\beta_2$ and the losers win many steps, then it is most likely that in fact the loser with the heaviest tail wins all those steps, i.e.
    $$\P\left(X_1(\infty)>x\,\Bigg|\, \sum_{i=1}^aX_i(\infty)>x, sMon_1^c\cap\ldots\cap sMon_a^c\right)\xrightarrow{x\to\infty}1$$
    for any $a<A$, consistent with the idea of Sanov's Theorem.
\end{example}

Moreover, Theorem \ref{thm: loserSuperlin} can be used to generalize previous results from \cite{Zhu, Cotar} concerning the time of monopoly, which is defined as
\begin{equation}\label{eq: MonTimeDef}
N_{mon}\coloneqq\min\{n\ge1 \colon \exists i\in[A]\,\forall\,m\ge n\colon X(m)-X(m-1)=e^{(i)}\}\,,  
\end{equation}
i.e. the index of the last step won by a loser (plus one). Using Theorem \ref{thm: loserSuperlin}, we can calculate the tail distribution of $N_{mon}$ for polynomial feedback  $F_i(k)=\alpha_ik^{\beta_i}$ with $\beta_i>1$ and $\alpha_i>0$, $i\in[A]$. Assume $\beta_2\le \beta_j$ for all $j\ge2$, i.e.\ among the losers agent 2 has the weakest feedback.
Then the following can be shown (under mild further assumptions), conditional on the first agent being the winner:

\begin{enumerate}
    \item If $\beta_2\ge \beta_1-1$, then we have
    \begin{equation}\label{eq ex1}
        \P(N_{mon}>n\,|\,sMon_1)\asymp n^{-(\beta_2 -1)(\beta_1 -1)/\beta_2}\quad\text{for }n\to\infty\,.
    \end{equation}
     \item If $1<\beta_2\le \beta_1-1$, then we have
        \begin{equation}\label{eq ex2}
    \P(N_{mon}>n\,|\,sMon_1)\asymp n^{-(\beta_2 -1)}\quad\text{for }n\to\infty\,.
    \end{equation}
\end{enumerate}

\begin{figure}
  \centering
  \subfloat{\includegraphics[width=0.7\linewidth]{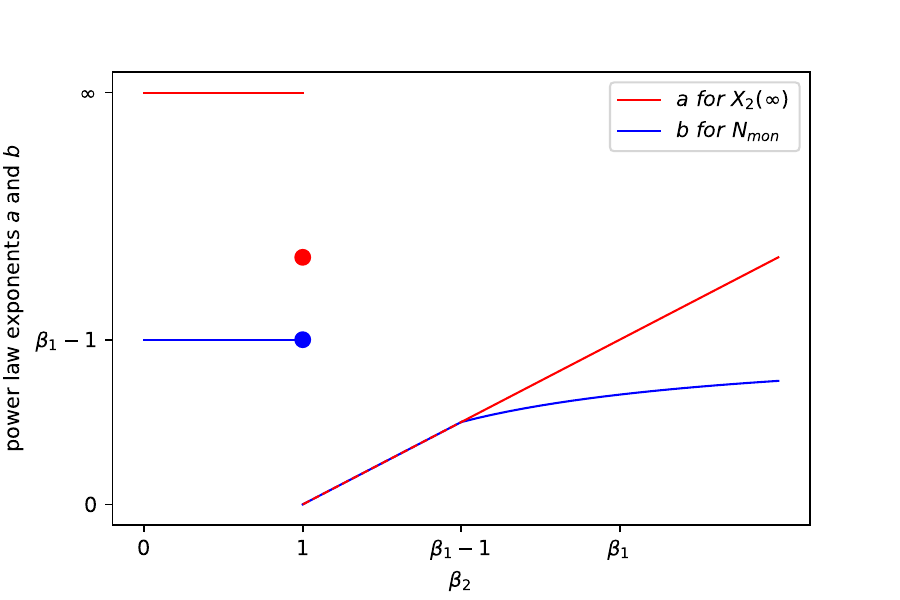}}
  \caption{Polynomial feedback for $A=2$. The power law exponents $a, b>0$ of $X_1(\infty)$ and $N_{mon}$ for fixed feedback of the winner $F_1(k)=k^{\beta_1}$, $\beta_1>2$, and varying feedback of the loser $F_2(k)=k^{\beta_2}$, $\beta_2>0$, i.e. $\P(X_2(\infty)>x\,|\, sMon_1)\asymp x^{-a}$ for $x\to\infty$ and  $\P(N_{mon}>n\,|\, sMon_1)\asymp x^{-b}$ for $n\to\infty$. $a=\infty$ corresponds to tails lighter than a power law. Theorem \ref{thm: loserSuperlin} and Corollary \ref{cor: MonTimeSuperlin} cover the case $\beta_2 >1$, and Theorem \ref{thm: loserSublin} and Corollary \ref{cor: MonTime} cover the case $\beta_2 \leq 1$. The transition at $\beta_2 =1$ is also discussed in Section \ref{sec summary}. Note that $a\geq b$ since on the event $sMon_1$ we have $N_{mon} \geq X_2 (\infty )$.} 
  \label{figure: ExponentPlot}
\end{figure}

A proof is presented in Appendix \ref{sec: time} as well as a discussion of exponential feedback. As opposed to the wealth of losers, the tail of the time of monopoly depends also on the feedback of the winner. Note that the tail behaviour of $N_{mon}$ is independent of the losers with stronger feedback $F_j$, $j\geq 3$. Figure \ref{figure: ExponentPlot} illustrates this situation for $A=2$. This result grants some new interesting insights in the large deviation behaviour of asymmetric generalized Pólya urns. For example, when the monopoly sets in late, this happens in different ways depending on the feedback. For simplicity, we assume $A=2$ and condition on agent $1$ to be the winner. The following is established rigorously in Corollary \ref{cor: ShareMonTime}:

\begin{enumerate}
    \item If the feedback of the winner is significantly stronger than the feedback of the loser (i.e. $1<\beta_2\le \beta_1-1$), then the tail of $N_{mon}$ is independent of $\beta_1$ and equals the one of the wealth of the loser $X_2 (\infty )$ up to a constant prefactor ($a=b$ in Figure \ref{figure: ExponentPlot}).
    So the loser wins most steps at the beginning (even more than the winner), when suddenly the strong monopoly sets in and the loser does not win any further steps.
    
    \item If the feedback of the winner is at most slightly stronger than the feedback of the loser (i.e. $\beta_2 \geq \beta_1 -1$), then the loser can still win some late steps when the advantage of the winner is already large ($a>b$ in Figure \ref{figure: ExponentPlot}).
    
    \item At the point of transition $\beta_2=\beta_1 -1$, we observe an intermediate regime, where the process is balanced at the beginning, when suddenly the strong monopoly sets in and the monopolist wins all further steps.
\end{enumerate}

\section{Sub-linear agents}\label{sec: LoserSublin}

So far, we mostly considered agents satisfying the monopoly condition (\ref{eq: explosion}). Nevertheless, any system with at least one agent fulfilling (\ref{eq: explosion}) exhibits strong monopoly, even if some agents do not satisfy (\ref{eq: explosion}). The latter agents are almost surely not the monopolist and hence, $X_i(\infty)$ is a finite random variable for them. The following theorem characterizes the corresponding tail distribution.

\begin{theorem}\label{thm: loserSublin}
    Assume that the set $M :=\{ i\in [A]:F_i \mbox{ satisfies }\ref{eq: explosion}\}$ of possible monopolists is non-empty but $1\not\in M$. Moreover, assume $\lim_{k\to\infty}F_1(k)=\infty$ and that all $F_i$, $i\in M$ are strictly monotone. Then
    $$-\log(\P(X_1(\infty)>x))\sim d\sum_{k=X_1(0)}^x \frac{1}{F_1(k)}\quad\text{for }x\to\infty\,,\quad\mbox{where}\quad d\coloneqq\sum_{i\in M} F_i(X_i(0))\ . $$
\end{theorem}

\begin{proof}
    Let $T\coloneqq\min_{i\in M}T_i$ be the smallest explosion time of the possible monopolists in the exponential embedding. Moreover, write $S(x)\coloneqq\sum_{k=X(0)}^x\tau_1(k)$ and denote by $\mu_{S(x)} =\P [S(x)\in \cdot]$ the law of $S(x)$. Using the independence of $S(x)$ and $T$, we get:
    \begin{align*}
        \P(X_1(\infty)> x)&=\P(\Xi_1(T)> x)=\P\left(S(x)<T\right)=\int_0^\infty \P(s<T)\,d\mu_{S(x)}(s)\\
        &=\int_0^\infty \prod_{i\in M}\P(T_i>s)\,d\mu_{S(x)}(s)
    \end{align*}
    Next, (\ref{eq: Ttail}) implies
    \begin{equation}\label{eq: LoserSublinEst}
        e^{-F_i(X_i(0))s}=\P(\tau_i(X(0))>s)<\P(T_i>s)<ce^{-F_i(X_i(0))s}
    \end{equation}
    for a constant $c>0$. Hence:
    $$\E\left[e^{-dS(x)}\right]<\P(X_1(\infty)> x)<c^{\#M}\E\left[e^{-dS(x)}\right]$$
    Then the following Lemma \ref{lemma: MGFbound} and 
    $$\lim_{k\to\infty}F_1(k)=\infty\quad\Longrightarrow\quad\lim_{k\to\infty}\frac{\sum_{l=1}^k\frac{1}{F_1(k)}}{\sum_{l=1}^k\frac{1}{F_1(k)^2}}=\infty$$
    finally imply the claim.
\end{proof}

Note that estimate (\ref{eq: LoserSublinEst}) does in general not hold if there was some $k>X_i(0)$ such that $F_i(k)=F_i(X_i(0))$. Moreover, Theorem \ref{thm: loserSublin} does not hold for non-diverging feedback, since e.g. for constant $F_i(k)=1$ we have
$$\E\left[e^{-dS(x)}\right]=\exp\left(-\log(1+d)\sum_{k=X_i(0)}^x \frac{1}{F_i(x)}\right) =(1+d)^{X_i (0)-x-1}\,.$$

\begin{lemma}\label{lemma: MGFbound}
    Assume that $X_1,\ldots, X_k$ are independent random variables, which are exponentially distributed with parameters $\lambda_1,\ldots, \lambda_k>0$. The the moment generating function of $S\coloneqq X_1+\ldots+X_k$ is bounded by
    $$\exp\left\{-s\sum_{l=1}^k\frac{1}{\lambda_k}\right\}<\E\left[e^{-sS}\right]<\exp\left\{-s\sum_{l=1}^k\frac{1}{\lambda_k}+s^2\sum_{l=1}^k\frac{1}{\lambda_k^2}\right\}$$
    for all $s>0$.
\end{lemma}

\begin{proof}
    Given $\E [e^{-sX_l} ]=\frac{\lambda_l}{\lambda_l+s}$, we calculate:
    \begin{align*}
        \E\left[e^{-sS}\right]=\prod_{l=1}^k\E^{-sX_l}=\prod_{l=1}^k\frac{\lambda_l}{\lambda_l+s}=\exp\left\{\sum_{l=1}^k\left(\log(\lambda_l)-\log(\lambda_l+s)\right)\right\}
    \end{align*}
    Moreover, since $\frac{d}{ds}\log(s)=\frac{1}{s}$, we can estimate
    \begin{equation*}
        \log(\lambda_l+s)-\log(\lambda_l)\le \frac{s}{\lambda_s}\,,
    \end{equation*}
    and analogously
    \begin{align*}
        \log(\lambda_l+s)-\log(\lambda_l)\ge \frac{s}{\lambda_l+s}=\frac{s}{\lambda_l}+s\left(\frac{1}{\lambda_l+s}-\frac{1}{\lambda_l}\right)\ge \frac{s}{\lambda_l}-\frac{s^2}{\lambda_l^2}\,,
    \end{align*}
    which completes the proof.
\end{proof}

A closer look at the proof reveals the following refinement of Theorem \ref{thm: loserSublin}
\begin{align*}
   0<\log(\P(X_1(\infty)>x))+d\sum_{k=X_1(0)}^x \frac{1}{F_1(k)}<d^2\sum_{k=X_1(0)}^x \frac{1}{F_1(k)^2}+const.
\end{align*}
for all $x\ge X_1(0)$. Let us discuss some interesting examples.

\begin{example}\label{example: loserWealthSublin}
In the situation of Theorem \ref{thm: loserSublin}, we consider various choices of the sub-linear feedback function $F_1$.
\begin{enumerate}
     \item Let $F_1(k)=k\log k$. Then 
     \begin{equation}\label{eq: loserTailxlogx}
         \P(X_1(\infty)>x)\asymp  (\log x)^{-d}\quad\text{for } x\to\infty\,.
     \end{equation}
     Hence, the tail decay is logarithmic and depends on $d$.
       \item Let $F_1(k)=k(\log k)^\beta$ for $\beta<1$. Then
       \begin{equation}\label{eq: loserTailBeta<1}
           \P(X_1(\infty)>x)\asymp e^{-\frac{d}{1-\beta}(\log x)^{1-\beta}}\quad\text{for }x\to\infty\,.
       \end{equation}
      Hence, $X_1(\infty)$ has a heavy tailed distribution, which is heavier for larger $\beta$. In particular, for the linear case $\beta=0$ we find a power law distribution with exponent depending on $\beta$ and the initial values included in $d$, i.e.
      \begin{equation}\label{eq: loserTailLinear}
          \P(X_1(\infty)>x)\asymp x^{-d}\quad\text{for }x\to\infty\,.
      \end{equation}
    \item Let $F_1(k)=k^\beta$ for $\frac12<\beta<1$. Then  
    $$\P(X_1(\infty)>x)\asymp e^{-\frac{d}{1-\beta}x^{1-\beta}}\quad\text{for }x\to\infty\,.$$
    Hence, the tail of $X_1(\infty)$ is lighter than a power law, but heavier than an exponential. Moreover, the tail is lighter the smaller $\beta$ is.
    \item Let $F_1(k)=k^{\frac{1}{2}}$. Then 
    $$e^{-\frac{d}{1-\beta}x^{1-\beta}}\prec \P(X_1(\infty)>x)\prec e^{-\frac{d}{1-\beta}x^{1-\beta}}x^{d^2}\quad\text{for }x\to\infty\,.$$
    The same observation as in 3. holds here.
    \item Let $F_1(k)=k^\beta$ for $0<\beta<\frac12$. Then  
    $$e^{-\frac{d}{1-\beta}x^{1-\beta}}\prec \P(X_1(\infty)>x)\prec e^{-\frac{d}{1-\beta}x^{1-\beta}+\frac{d^2}{1-2\beta}x^{1-2\beta}}\quad\text{for }x\to\infty$$
    consistent with 3. and 4..
    \item Let $F_1(k)=\lambda>0$ be constant. Then Theorem \ref{thm: loserSublin} is not applicable, but $\Xi_1$ is a homogeneous Poisson process. Hence, $X_1(\infty)$ has a Poisson distribution. 
\end{enumerate}
\end{example}

\begin{figure}
  \centering
  \subfloat[$X(0)=(1, 1)$ and varying $F_1$]{\includegraphics[width=0.5\linewidth]{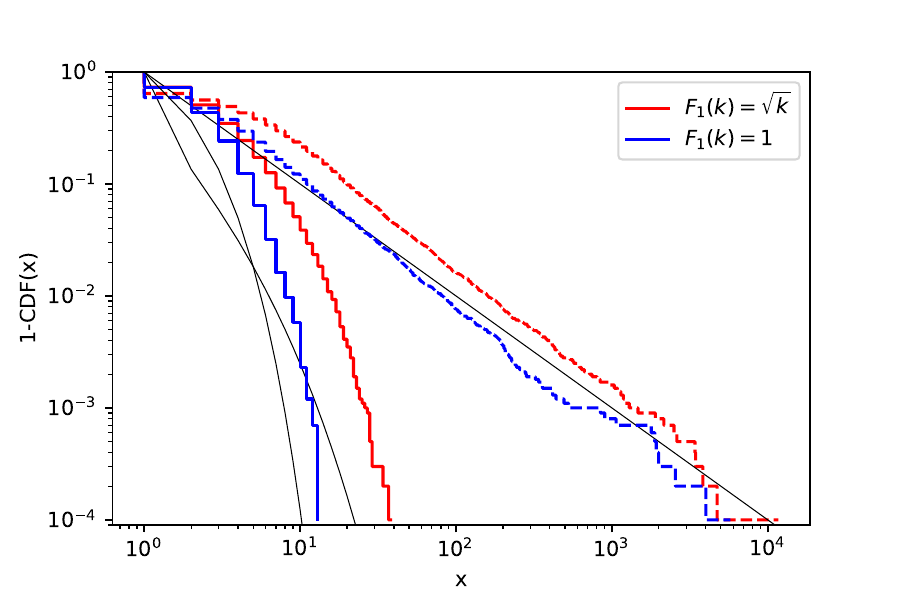}}
  \subfloat[$F_1(k)=k$ and varying $X(0)$]{\includegraphics[width=0.5\linewidth]{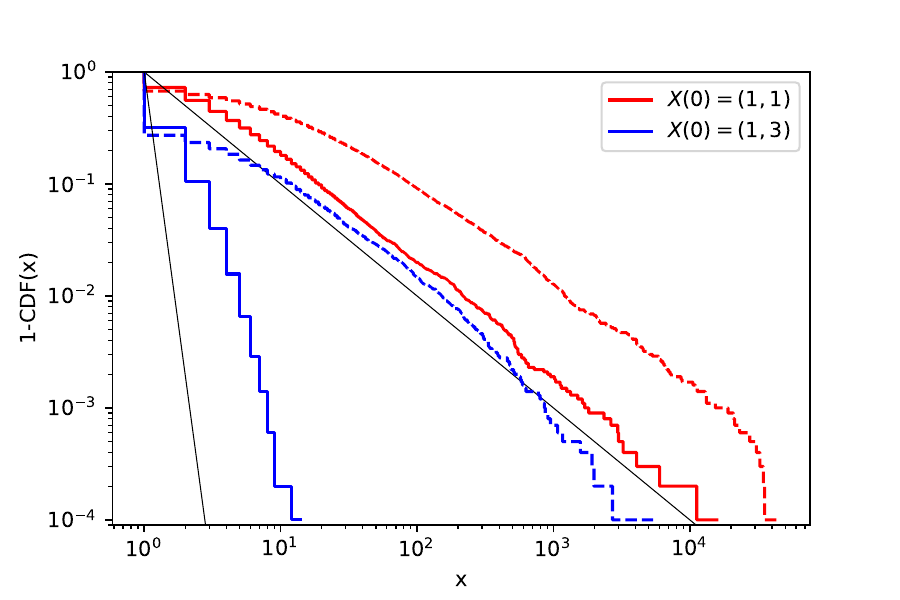}}
  \caption{Empirical distribution of $X_1(\infty)$ (full line) and $N_{mon}$ (dotted line) for $A=2$, $F_2(k)=k^2$ and different $F_1$ and $X(0)$. The black lines show the tail decay predicted by Theorem \ref{thm: loserSublin} and Corollary \ref{cor: MonTime}. Note that these predicted tails are only exact up to constants, i.e. up to a parallel shift in the plot. 10,000 simulations were executed each.}
  \label{figure: LoserSimulation2}
\end{figure}

 Remarkably, the tail decay of $X_i(\infty)$ depends significantly on initial values of super-linear agents for feedback $F_1$ close to the transition (\ref{eq: explosion}), in contrast to super-linear $F_1$ in Theorem \ref{thm: loserSuperlin}. Moreover, due to the assumption of monotonicity, the tail does not depend on the feedback function of other agents up to multiplicative constants. These findings are illustrated by the simulation shown in Figure \ref{figure: LoserSimulation2}. As mentioned in \eqref{eq ex1} and \eqref{eq ex2} and explained in Appendix \ref{sec: timeSublin}, the time when the monopoly occurs in the situation of this section has a power-law distribution (for homogeneous feedback) and the exponent depends on the feedback of the winner.\\

\section{Summary and discussion\label{sec summary}}

By combining Theorem \ref{thm: loserSuperlin} and Theorem \ref{thm: loserSublin}, we gain a precise understanding of the "loser paradox" mentioned in the introduction. This seeming paradox is that the tail of $X_i(\infty)$ (if necessary conditioned on $sMon_i^c$) is heaviest, when the feedback of agent $i$ is close to the transition (\ref{eq: explosion}). This observation has already been quantified by Example \ref{example: loserWealthSublin} and Example \ref{example: birth}. At the point of transition for almost linear feedback, these two examples are consistent. To see that, consider again $F_i(k)=k(\log k)^\beta$ and condition on agent $i$ to be a loser. Recall that throughout this paper we assumed that at least one other agent satisfies the monopoly condition (\ref{eq: explosion}).
\begin{enumerate}
    \item For $\beta\in(-\infty, 0)$, the tail of $X_i(\infty)$ is heavier than exponential, but lighter than a power law, see (\ref{eq: loserTailBeta<1}). This does also hold for $F_i(k)=k^{\beta'}$ with $\beta'\in (0, 1)$ (3. to 5. in Example \ref{example: loserWealthSublin}).
    \item In the linear case $\beta=0$, we have a power law tail $x^{-d}$ with exponent depending on initial values of the super-linear agents, see (\ref{eq: loserTailLinear}).
    \item  For $\beta\in(0, 1)$, the tail of  $X_i(\infty)$ is even heavier than a power law, but the decay is faster than logarithmic, see (\ref{eq: loserTailBeta<1}). The tail weight is increasing in $\beta<1$ and depends on the initial values.
    \item For $\beta\ge 1$, the tail of  $X_i(\infty)$ decays logarithmically,  i.e. $\P(X_1(\infty)>x)\asymp  (\log x)^{-d}$ for $\beta=1$, where $d$ depends on initial values, and $\P(X_1(\infty)>x)\asymp  (\log x)^{1-\beta}$ for $\beta>1$, see (\ref{eq: loserTailxlogx}) and (\ref{eq: loserTailbeta>1}). The tail weight is now decreasing in $\beta>1$, which is consistent with the power law distribution of $X_i(\infty)$ for $F_i(k)=k^{\beta'}$, $\beta'>1$, see (\ref{eq: loserTailSuperlin}). 
\end{enumerate}

As a main achievement of this paper, we have generalized previous results in \cite{Zhu, Oliveira, Cotar} about the tail of losing agents in non-linear Poly\'a urn models that exhibit strong monopoly. As a consequence, we get the above picture, which shows a non-monotonicity: The losers with feedback close to linear are most likely to win many steps. As an important application of these findings, we also extended results from \cite{Zhu, Cotar} on the time $N_{mon}$ of monopoly to systems with asymmetric feedback (Corollary \ref{cor: MonTimeSuperlin} and Corollary \ref{cor: MonTime}), allowing new insight in the large deviation behavior of asymmetric Pólya urns (Corollary \ref{cor: ShareMonTime}). On top of that, we provide a precise description of the positive tail dependence between several losing agents (Theorem \ref{thm: LoserCorrelation}), which does not vanish even for diverging system size $A\to\infty$ (Appendix \ref{sec: aDependence}). Nevertheless, in the limit $A\to\infty$ a fixed set of agents does not win a single step with high probability (Theorem \ref{thm: limA}).

\begin{footnotesize}
\bibliography{Birth_processes}
\bibliographystyle{plainnat}
\end{footnotesize}

\appendix

\section{Systems with many agents}

\subsection{General behaviour for $A\to\infty$}\label{sec: manyAgents}

In some applications of the generalized Pólya urn model, the number of agents is typically large, e.g. when then wealth distribution within a nation is modelled. Although our results on the wealth of agents hold for general finite $A$, we now take a closer look at the limiting case $A\to\infty$.\\

For that, we fix a sequence $F_1, F_2,\ldots$ of feedback functions and a sequence $X_1(0), X_2(0),\ldots\in\N$ of initial values. Denote by $(X^{(A)}(n))_{n\ge0}$ a generalized Pólya urn with $A$ agents and feedback functions $F_1,\ldots, F_A$ and initial values $X^{(A)}(0)=(X_1(0),\ldots, X_A(0))$. In addition, define
$$X^{(A)}_i(\infty)\coloneqq\lim_{n\to\infty}X_i^{(A)}(n)\in\N\cup\{\infty\}.$$
The approach is based on the exponential embedding described in Section \ref{sec: applicatonPolya}. So, let $\Xi_1, \Xi_2,\ldots$ be a sequence of independent birth processes with initial values $X_1(0), X_2(0),\ldots$ and rate functions $F_1, F_2,\ldots$. We denote by $\tau_i(k)$ the sojourn times of $\Xi_i$ and by  $T_i\coloneqq\sum_{k=X_i(0)}^\infty\tau_i(k)\in(0, \infty]$ the explosion time of $\Xi_i$. In the following, $\P$ denotes the law of the sequence of birth processes and for each $A\geq 1$ the process $(X^{(A)}(n))_{n\ge0}$ is then defined as in Section \ref{sec: applicatonPolya}.

First, we consider the case when infinitely many $F_i$ satisfy the monopoly condition (\ref{eq: explosion}). Theorem \ref{thm: loserSuperlin} describes the tail distribution of a loser among those agents. Although the tail decay rate does not depend on $A$, the constant prefactor vanishes for $A\to\infty$. Indeed, a typical agent does not even win a single step in large systems.

\begin{theorem}\label{thm: limA}
    Assume that there is an infinite subset $B\subset\N$, such that
    \begin{equation}\label{eq: AlimAssumption}
        \sum_{k=0}^\infty\max_{i\in B}\frac{1}{F_i(k)}<\infty.
    \end{equation}
   Then: 
    \begin{enumerate}
        \item The convergence
        $$\lim_{A\to\infty}\P\left(X^{(A)}_1(\infty)=X_1(0)\right)=1$$
        holds, i.e. $X^{(A)}_1(\infty)\xrightarrow{A\to\infty}X_1(0)$ in distribution.
        \item If in addition $\limsup_{i\to\infty}F_i(X_i(0))<\infty$, then 
        $$ \frac{1}{A}\#\{i\in[A]\colon X_i^{(A)}(\infty)=X_i(0)\}\xrightarrow{A\to\infty}1\quad\text{in distribution.}$$
        \item If $\liminf_{i\to\infty}F_i(X_i(0))>0$ and $\lim_{i\to\infty}\frac1iF_i(X_i(0)+1)=0$, then
         $$\#\{i\in[A]\colon X_i^{(A)}(\infty)\ne X_i(0)\}\xrightarrow{A\to\infty}\infty\quad\text{in distribution.}$$
    \end{enumerate}
\end{theorem}

In other words: In the limit $A\to\infty$, almost all agents do not win a single step. Nevertheless, the absolute number of agents winning at least one step is unbounded. Note that (\ref{eq: AlimAssumption}) implies that infinitely many agents fulfill (\ref{eq: explosion}). Obviously, (\ref{eq: AlimAssumption}) is satisfied when infinitely many $F_i$ are equal and satisfy (\ref{eq: explosion}), a typical counterexample would be $F_i(k)=k^{1+\frac{1}{i}}$. The assumption (\ref{eq: AlimAssumption}) can be replaced by any other condition, that ensures that $\min_{i=1,\ldots, A}T_i\xrightarrow{A\to\infty}0$ in distribution. 

\begin{lemma}\label{lem: limA}
    Assume that (\ref{eq: AlimAssumption}) holds. Then:
\begin{equation}
\min_{i=1,\ldots, A}T_i\xrightarrow{A\to\infty}0\quad\text{in distribution.}
\end{equation}
\end{lemma}

\begin{proof}
We fix $\epsilon>0$ and show that $\P\left(\min_{i=1,\ldots, A}T_i>\epsilon\right)\xrightarrow{A\to\infty}0$. For that, we take a sequence $(\tilde\tau(k))_k$ of independent random variables, where $\tilde\tau_i(k)$ is exponentially distributed with expectation $\max_{i\in B}\frac{1}{F_i(k)}$. Then $\tilde T\coloneqq\sum_{k=1}^\infty\tilde\tau(k)$ is almost surely finite due to assumption (\ref{eq: AlimAssumption}). Moreover, $\tilde T$ is stochastically bigger than all $T_i$ with $i\in B$. Using the independence of the $T_i$, we can now estimate
\begin{align*}
    \P\left(\min_{i=1,\ldots, A}T_i>\epsilon\right)&=\prod_{i=1}^A\P(T_i>\epsilon)\le\prod_{i\in B\atop i\le A}\P(T_i>\epsilon)\le\prod_{i\in B\atop i\le A}\P(\tilde T>\epsilon)\xrightarrow{A\to\infty}0
\end{align*}
since $B$ is infinite and $\P(\tilde T>\epsilon)<1$ as a sum of exponential random variables.
\end{proof}

Now, the independence property of the exponential embedding allows a canonical coupling of the processes $X^{(A)}$, which helps to prove Theorem \ref{thm: limA}.

\begin{proof}[Proof of Theorem \ref{thm: limA}]
    All independent random variables $\tau_i(k), i\ge1, k\ge1$ are defined on the same probability space with distribution $\P$. The sequence $(\min_{i=2,\ldots,A}T_i)_A$ is decreasing in $A$ and hence almost surely convergent with limit zero (due to Lemma \ref{lem: limA}). Now, fix a realisation. Then $\tau_1(X_1(0))>\min_{i=2,\ldots,A}T_i$ holds for $A$ large enough, i.e. $X_1^{(A)}(\infty)=X_1(0)$. Thus, $X_1^{(A)}(\infty)$ converges to $X_1(0)$ almost surely for $A\to\infty$, which implies part 1 of Theorem \ref{thm: limA}.

    For part 2, we observe that
    $$T_{2:A}\coloneqq\min\left\{T_i\colon i\in[A]\text{ and }T_i\ne\min_{j=1,\ldots,A}T_j\right\}\xrightarrow{A\to\infty}0\quad\text{almost surely}$$
    holds, too. Due to dominated convergence, we even have $\E e^{-cT_{2:A}}\xrightarrow{A\to\infty}1$ for any constant $c>0$. In addition, define $T_{[A]\backslash\{i\}}\coloneqq\min_{j\in[A]\backslash\{i\}}T_j$, which obviously satisfies $T_{[A]\backslash\{i\}}\le T_{2:A}\to0$. Now we are prepared for the final calculation:
    \begin{align*}
        \E&\left(\frac{\#\{i\in[A]\colon X_i^{(A)}(\infty)=X_i(0)\}}{A}\right)=\E\left(\frac1A\sum_{i=1}^A\mathds{1}_{X_i^{(A)}(\infty)=X_i(0)}\right)=\frac1A\sum_{i=1}^A\P\left(X_i^{(A)}(\infty)=X_i(0)\right)\\
        &=\frac1A\sum_{i=1}^A\E\left[\P\left(X_i^{(A)}(\infty)=X_i(0)\,\big|\,T_{[A]\backslash\{i\}}\right)\right]=\frac1A\sum_{i=1}^A\E\left[\P\left(\tau_i(X_i(0))>T_{[A]\backslash\{i\}}\,\big|\,T_{[A]\backslash\{i\}}\right)\right]\\
        &=\frac1A\sum_{i=1}^A\E \exp\left(-F_i(X_i(0))T_{[A]\backslash\{i\}}\right)\ge\frac1A\sum_{i=1}^A\E \exp\left(-F_i(X_i(0))T_{2:A}\right)\\ 
        &\ge\E \exp\left(-const.T_{2:A}\right)\xrightarrow{A\to\infty}1
    \end{align*}
    This is sufficient for 2. since $\frac{1}{A}\#\{i\in[A]\colon X_i^{(A)}(\infty)=X_i(0)\}\le1$.

    For part 3, we take any $k\in\N$ and show, that the probability, that the first $k$ steps are won by $k$ different agents, converges to one for $A\to\infty$:
    \begin{align*}
        \P&\left(\#\{i\in[A]\colon X_i^{(A)}(\infty)\ne X_i(0)\}\ge k\right)\\
        &\ge\P\left(X(n)-X(n-1)\ne X(m)-X(m-1)\text{ for all } n, m=1,\ldots, k,\,n\ne m\right)\\
        &=\prod_{n=1}^k\P\Big(\forall l<n\colon X(n)-X(n-1)\ne X(l)-X(l-1)\\
        &\qquad\qquad\Big|\,\forall l<m<n\colon X(m)-X(m-1)\ne X(l)-X(l-1)\Big)\\
        &\ge\prod_{n=1}^k\left(1-\frac{n\max_{i=1,\ldots,A}F_i(X_i(0)+1)}{(A-n)\min_{i=1,\ldots,A}F_i(X_i(0))}\right)\xrightarrow{A\to\infty}1
    \end{align*}
\end{proof}

The same coupling argument also implies, that all agents gain less the more agents there are in the system, i.e.
$$\P\left(X_i^{(A)}(\infty)\ge k\right)>\P\left(X_i^{(A+1)}(\infty)\ge k\right)\,.$$

Another important implication of of Theorem \ref{thm: limA} is a trivial asymptotic independence of agents in large systems.

\begin{corollary}\label{cor: trivialIndependence}
    Assume that (\ref{eq: AlimAssumption}) holds. Then we have for any $A_0\in\N$ and any Borel-sets $B_1,\ldots, B_{A_0}$ that
    $$\lim_{A\to\infty}\P(X_1^{(A)}(\infty)\in B_1,\ldots,X_{A_0}^{(A)}(\infty)\in B_{A_0})=\lim_{A\to\infty}\prod_{i=1}^{A_0} \P(X_i^{(A)}(\infty)\in B_i)\,.$$
\end{corollary}

In order to illustrate Theorem \ref{thm: limA}, we executed 100 simulations of a process with $A=100$ agents and $F_1(k)=\ldots=F_{100}(k)=k^2$. We interrupted the process after 100,000 steps. On average, 30.14 agents won at least one step of the process, with a minimum of 16 and a maximum of 45 agents. In a simulation with $A=10,000$ agents, interrupted after $10,000,000$ steps, 2010 agents won at least one step. Hence, the convergence in part 2 of Theorem \ref{thm: limA} can be considered as slow.

The following examples show that the assumptions of Theorem \ref{thm: limA} cannot be omitted.

\begin{example}
    \begin{enumerate}
        \item Assume that $\sum_{i=1}^\infty F_i(X_i(0))<\infty$, for example set $F_i(k)=2^{-i}k^\beta,\,\beta>1$ and  $X_i(0)=1$ for all $i\in\N$, such that (\ref{eq: AlimAssumption}) does not hold. Then the probability that agent 1 wins the first step of the process is $$\frac{F_1(X_1(0))}{\sum_{i=1}^A F_i(X_i(0))}\ge\frac{F_1(X_1(0))}{\sum_{i=1}^\infty F_i(X_i(0))}>0.$$ Thus, $\P\left(X^{(A)}_1(\infty)=X_1(0)\right)\ge const.$ does not converge to zero for $A\to\infty$
        \item Set $X_i(0)=1$ for all $i\in\N$. In order to construct an appropriate sequence of feedback functions, take a sequence $(c_i)_i\in(0,1)$, such that $\prod_{i=1}^\infty c_i>0$, e.g. $c_i=e^{-(i^{-2})}$. Moreover, define another sequence $(a_i)_i\in(0,\infty)$ by the recursion formula $a_{i}\ge \frac{c_i}{1-c_i}\left(\sum_{j=1}^{i-1}a_j+i\right),\,a_1=1$. Then consider the following sequence of feedback functions:
        $$F_i(k)=\begin{cases}
            a_i &\text{if }k=1\\
            1 &\text{if }k=2\\
            k^2 &\text{if }k\ge3
        \end{cases}$$
        Note that this sequence fulfills (\ref{eq: AlimAssumption}), but $\limsup_{i\to\infty}F_i(X_i(0))=\infty$. For any $A$, the probability that agent $i$ wins the $(A-i+1)$-th step of the process for all $i\ge\frac{A}{2}$ is given by
        $$\prod_{i=\lceil A/2\rceil}^A\frac{a_{i}}{\sum_{j=1}^ia_j+A-i}\ge\prod_{\lceil A/2\rceil}^A\frac{a_{i}}{\sum_{j=1}^ia_j+2i-i}\ge\prod_{i=\lceil A/2\rceil}^Ac_i\ge\prod_{i=1}^\infty c_i>0.$$
        Hence, $\P\left(\frac{1}{A}\#\{i\in[A]\colon X_i^{(A)}(\infty)=X_i(0)\}<\frac12\right)$ is bounded away from zero (uniformly in $A$).
    \end{enumerate}
\end{example}

Nevertheless, condition (\ref{eq: AlimAssumption}) is fulfilled in most generic situations as the following example underlines. Note that (\ref{eq: AlimAssumption}) does not depend on the initial values $X_i(0)$.

\begin{example}
\begin{enumerate}
    \item Let $F_i(k)=\alpha_ik^{\beta_i}$ for $\beta_i>0$. Then (\ref{eq: AlimAssumption}) is fulfilled if $\limsup_{i\to\infty}\min\{\alpha_i, \beta_i-1\}>0$.
    \item Let $F_i(k)=\alpha_i F(k)$ for a function $F$ satisfying (\ref{eq: explosion}). In addition, assume that the $\alpha_i>0$ are a realisation of an independent and identically distributed sequence of random variables. Then (\ref{eq: AlimAssumption}) is fulfilled, too.
\end{enumerate}
\end{example}

\subsection{Large systems with linear or sub-linear feedback}

In this subsection (and only here), we take a look at large systems where no agent satisfies the monopoly condition (\ref{eq: explosion}), such that all agents win infinitely many steps. We keep the notations from the previous Subsection \ref{sec: manyAgents} and denote by $\chi_i^{(A)}(n) =X_i^{(A)}(n) /\sum_{j\in [A]} X_j^{(A)}(n)$ the process of shares corresponding to $X_i^{(A)}(n)$. For the classical Pólya urn with linear feedback, it is well known that the long-time market shares
$$\chi^{(A)}_i(\infty)\coloneqq\lim_{n\to\infty}\chi_i^{(A)}(n)\in[0, 1]$$
exist and that $\left(\chi_1^{(A)}(\infty),\ldots,\chi_A^{(A)}(\infty)\right)$ has a Dirichlet distribution (see e.g. \cite{Freedman}). The following Theorem states that $X_i^{(A)}(n)$ is approximately exponentially distributed with mean $\frac{n}{A}+1$, when $A$ and $n$ is large and when all agents start with the same wealth at time zero. Moreover, any fixed number of agents is asymptotically independent for $A\to\infty$. Thus, the empirical common wealth distribution of all agents is asymptotically iid exponential and the share of the richest agent is approximately given by  $(\log A) /A$.

\begin{theorem}\label{thm: asymptoticExp}
    Let $F_1(k)=F_2(k)=\ldots=k$ and $X_1(0)=X_2(0)=\ldots=1$. Moreover, fix $A_0\in\N$ and $x_1,\ldots,x_{A_0}>0$. Then the following holds:
    \begin{align}\label{thm: asymptoticExp1}
         \P\left(A\cdot\chi_1^{(A)}(\infty)\le x_1,\ldots, A\cdot\chi_{A_0}^{(A)}(\infty)\le x_{A_0}\right)\xrightarrow{A\to\infty}\prod_{i=1}^{A_0}\left(1-e^{-x_i}\right)
    \end{align}:
   \begin{align}\label{thm: asymptoticExp2}
       \frac{A}{\log(A)}\max_{i=1,\ldots, A}\chi^{(A)}_i(\infty)\xrightarrow{A\to\infty}1\quad\text{in distribution}
   \end{align}
   \begin{align}\label{thm: asymptoticExp3}
       \sup_{x\ge0}\left|\frac{1}{A}\sum_{i=1}^A\mathds{1}_{\{A\chi^{(A)}_i(\infty)> x\}}-e^{-x}\right|\xrightarrow{A\to\infty}0 \quad\text{in distribution}
   \end{align}
\end{theorem}

\begin{proof}
    For (\ref{thm: asymptoticExp1}), suppose that $A$ is large enough, such that $\sum_{i=1}^{A_0}\frac{x_i}{A}<1$. It is well known, that $\left(\chi_1^{(A)}(\infty),\ldots,\chi_A^{(A)}(\infty)\right)$ is uniformly distributed on the simplex $\Delta_{A-1}$ for all $A$. Then:
\begin{align*}
    &\P\left(A\cdot\chi_1^{(A)}(\infty)\le x_1,\ldots, A\cdot\chi_{A_0}^{(A)}(\infty)\le x_{A_0}\right)= \P\left(\chi_1^{(A)}(\infty)\le \frac{x_1}{A},\ldots, \chi_{A_0}^{(A)}(\infty)\le \frac{x_{A_0}}{A}\right)=\\
    &=\int_0^{x_1/A}\cdots\int_0^{x_{A_0}/A}\int_0^{1-y_1-\ldots-y_{A_0}}\cdots\int_0^{1-y_1-\ldots-y_{A-1}}(A-1)!\,dy_A\ldots dy_1\\
    &=(A-1)!\int_0^{x_1/A}\cdots\int_0^{x_{A_0}/A}\int_0^{1-y_1-\ldots-y_{A_0}}\cdots\int_0^{1-y_1-\ldots-y_{A-2}}(1-y_1-\ldots-y_{A-1})\,dy_{A-1}\ldots dy_1\\
    &=\frac{(A-1)!}{2}\int_0^{x_1/A}\cdots\int_0^{x_{A_0}/A}\int_0^{1-y_1-\ldots-y_{A_0}}\cdots\int_0^{1-y_1-\ldots-y_{A-3}}(1-y_1-\ldots-y_{A-2})^2\,dy_{A-2}\ldots dy_1\\
    &=\frac{(A-1)!}{(A-A_0)!}\int_0^{x_1/A}\cdots\int_0^{x_{A_0}/A}(1-y_1-\ldots-y_{A_0})^{A-A_0}\,dy_{A_0}\ldots dy_1\\
    &=\frac{(A-1)!}{(A-A_0+1)!}\int_0^{x_1/A}\cdots\int_0^{x_{A_0-1}/A}(1-y_1-\ldots-y_{A_0-1})^{A-A_0+1}\\
    &\qquad -(1-y_1-\ldots-y_{A_0-1}-\frac{x_{A_0}}{A})^{A-A_0+1}\,dy_{A_0-1}\ldots dy_1\\
    &=\sum_{\{i_1,\ldots,i_k\}\subset[A_0]}(-1)^{k}\left(1-\frac{x_{i_1}}{A}-\ldots-\frac{x_{i_k}}{A}\right)^A=\\
    &\xrightarrow{A\to\infty}\sum_{\{i_1,\ldots,i_k\}\subset[A_0]}(-1)^{k}\exp\left(-\sum_{l=1}^kx_{i_l}\right)=\sum_{\{i_1,\ldots,i_k\}\subset[A_0]}(-1)^{k}\prod_{l=1}^ke^{-x_{i_l}}=\prod_{i=1}^{A_0}\left(1-e^{-x_i}\right)
\end{align*}
Next, (\ref{thm: asymptoticExp2}) follows from the fact that $\chi^{(A)}(\infty)\in\Delta_{A-1}$ has a Dirichlet distribution with parameter $X(0)$, which allows for the following representation from \cite[Theorem 1.1]{feng}: Let $Z_1,\ldots, Z_A$ be independent, exponentially distributed random variables with expectation $1$. Then
$$\left(\frac{Z_i}{Z_1+\ldots+Z_A}\right)_{i\in[A]}$$
has also a Dirichlet distribution with parameter $(1,\ldots, 1)$, such that (\ref{thm: asymptoticExp2}) is easy to check.

Finally, (\ref{thm: asymptoticExp3}) follows from \cite[Theorem 2.17]{feng}, which states that in the exponential embedding $e^{-t}\,\Xi_i(t)$ converges for $t\to\infty$ almost surely to an exponentially distributed random value for all $i\in[A]$. Hence, the left hand side of (\ref{thm: asymptoticExp3}) is the empirical distribution function of an independent sample of $A$ exponential distributions. Then the claim is a consequence of the Glivenko–Cantelli theorem.
\end{proof}

For $X(0)=(1, \ldots, 1)$ and large $A$, the shares $\chi^{(A)}(n)$ do barely change in each step, such that no agent wins a significant share on the long run. This can be fixed by replacing the state space of the process $X^{(A)}(n)$ by $(0, \infty)^A$, while keeping all definitions of Section \ref{sec: applicatonPolya} mutatis mutandis. If $X(0)=\left(\frac1A, \ldots,\frac1A\right)$, then $\chi^{(A)}(\infty)$ still has a Dirichlet distribution with parameter $X(0)$. According to \cite[Theorem 2.1]{feng}, $\chi^{(A)}(\infty)$ converges for $A\to\infty$ towards a Poisson-Dirichlet distribution (up to sorting).\\

Finally, we take a quick look at the sub-linear case. As derived in e.g. \cite{wir}, we know that the long-time market shares $\chi_i^{(A)}(\infty)$ are deterministic and independent of the initial values. In order to bring some randomness to our model, we assume that all agents have the same feedback function up to some random attractiveness parameter $\alpha_i$. When $A$ and $n$ is large, then the distribution of $X_i^{(A)}(n)$ is approximately determined by the distribution of the $\alpha_i$ up to scaling and a distortion depending on $F$. Moreover, the agents are asymptotically independent. The following proposition formalizes this idea.

\begin{proposition}
    Let $F_i(k)=\alpha_iF(k)$ for all $i, k\in\N$, where $F$ fulfills 
    $$\frac{F_i(k)}{k}\sum_{l=1}^k\frac{1}{F_i(l)}\xrightarrow{k\to\infty}c\in(0, \infty)\,.$$
    In addition, suppose that $(\alpha_i)_{i\in\N}$ is a sequence of independent, positive and identically distributed random variables with $\E\alpha_i^c\in(0, \infty]$. Fix $A_0\in\N$. Then:
    $$\left(A\cdot\chi_1^{(A)}(\infty),\ldots,A\cdot\chi_{A_0}^{(A)}(\infty)\right)\xrightarrow{A\to\infty}\left(\frac{\alpha_1^c}{\E\alpha_1^c},\ldots, \frac{\alpha_{A_0}^c}{\E\alpha_1^c}\right)\quad\text{in distribution.}$$
\end{proposition}

\begin{proof}
    From \cite{wir}, we know that $\chi_i^{(A)}(\infty)=\left(1+\alpha_i^{-c}\sum_{j\ne i}\alpha_j^c\right)^{-1}$ almost surely. Hence, 
    $$A\cdot \chi_i^{(A)}(\infty)=\left(\frac1A+\alpha_i^{-c}\frac1A\sum_{j\ne i}\alpha_j^c\right)^{-1}\xrightarrow{A\to\infty}\frac{\alpha_i^c}{\E\alpha_1^c}\quad\text{almost surely}$$
    due to the law of large numbers.
\end{proof}

In particular, if the $\alpha_i$ have a heavy tailed distribution, such that $\E\alpha_i^c=\infty$, then $A\cdot\chi_1^{(A)}(\infty)\xrightarrow{A\to\infty}0$. A possible interpretation of that is that the wealth of a typical agent grows slower than the whole economy. Vice versa, most of the wealth is hold by the most attractive agents.

Now assume that $\alpha_i$ has finite moments and that $F(k)=k^\beta,\,\beta<1$, such that $c=\frac{1}{1-\beta}$. Let $A$ and $n$ be large. If $\beta=0$, then the distribution of $X_i^{(A)}(n)$ coincides with the distribution of $\alpha_i$ up to scaling.  This means that any distribution of wealth can be obtained only by differences in the attractiveness of the agents and without increasing returns. When $\beta$ gets closer to one, then the distribution of $X_i^{(A)}(n)$ gets more unequal. This is consistent with the idea of weak monopoly for the linear case with different attractiveness parameters (see \cite{wir}).

Generalizations of the Pólya urn model to infinitely many agents have also been studied recently in \cite{Bandyopadhyay} or \cite{Janson}.

\subsection{Dependence of $c(A, a)$ on $A$ and $a$}\label{sec: aDependence}

 In Subsection \ref{subsec: loserCor}, we introduced the constant $c(A, a)$ as a measure for the tail dependence of the wealth of $a$ super-linear losers in a system of $A>a$ agents in total. We have already discussed that $c(A, a)>1$ in a symmetric situation, corresponding to a positive correlation of losers. This section examines the dependence of $c(A, a)$ on $A$ and $a$, starting with $A$.\\

  $c(A, a)$ can by computed numerically via formula (\ref{eq: cFormula}), since Lemma \ref{lemma: zhu} provides explicit expressions for $g_i$ and $G_i$ and, consequently, also for the density of $S$. Hence, the expectations in (\ref{eq: cFormula}) can be well approximated by Riemann sums. In the symmetric case, the density of $S$ is simply given by $(A-a)G(s)^{A-a-1}g(s)$, $s\ge0$. The following table shows approximations of $c(A, 2)$ for different feedback functions and system sizes in the symmetric case $F_1=\ldots=F_A=F$ and $X(0)=(1,\ldots, 1)$. The formula from Lemma \ref{lemma: zhu} was truncated after 100 summands. For the Riemann approximation of the expectation integral, we took a bandwidth of $0.0001$ and computed the integral on the interval $[0, 50]$ (c.f. Figure \ref{figure: ParetoFaktor})

\begin{center}
\begin{tabular}{ c | c c c c c }
c(A, 2) & $A=3$ & $A=10$ & $A=100$ & $A=10^3$ & $A=10^6$\\\hline
$F(k)=k^2$ & 1.121 & 1.227 & 1.427 & 1.565 & 1.754 \\
$F(k)=e^{k-1}$ & 1.130 & 1.218 & 1.374 & 1.480 & 1.634 \\
$F(k)=k(\log(e-1+k))^2$ & 1.141 & 1.254 & 1.460 & 1.597 & 1.779
\end{tabular}
\end{center}

The table confirms our finding $c(A, a)>1$ from Proposition \ref{prop: c}. Moreover, we observe that $c(A, 2)$ does not strongly depend on the feedback function, but tends to be larger for slowly increasing $F$. Remarkably, $c(A, 2)$ seems to be slightly increasing in $A$. Thus, the tail dependence of agents seems to be stronger in larger systems, as opposed to the classical Pólya urn with linear feedback, where one can show an asymptotic independence of agents for $A\to\infty$ (see Appendix \ref{sec: manyAgents}). Heuristically, this phenomenon may be explained as follows: First, we observe the implication
\begin{equation}\label{eq: T2(x)}
    X_2(\infty)\ge x_2\quad\Longrightarrow\quad S>T_2(x_2)\coloneqq\sum_{k=X_2(0)}^{x_2-1}\tau_2(k)
\end{equation}
and that the first passage time $T_2(x_2)$ is independent of $A$ and $S$. Since $S=S(A)$ is decreasing in $A$, the event $S>T_2(x_2)$ is less likely for larger $A$. Hence, conditioning 
the distribution of $S$ on the increasingly unlikely event $X_2(\infty)\ge x_2$ has a stronger effect for larger $A$. So $S$ is more likely to be untypically large 
and hence $X_1(\infty)$ is slightly more likely to be large as well. The following Proposition and Conjecture \ref{conj: c} underline this idea.

\begin{proposition}\label{prop: increasingDependence}
    Assume $X_1(0)=X_2(0)=\ldots$ and that $F_1=F_2=\ldots$ satisfy (\ref{eq: explosion}). Define $S=S(A)\coloneqq\min_{i=3,\ldots, A}T_i$. Then we have for all $s>0$ and $x_2> X_2(0)$
    $$\frac{\P_{A, 2}(S(A)>s\,|\, X_2(\infty)\ge x_2)}{\P_{A, 2}(S(A)>s)}\to \infty\quad\text{for }A\to\infty\,,$$
    but still
    $$\P_{A, 2}(S(A)>s\,|\, X_2(\infty)\ge x_2)\to0\quad\text{for }A\to\infty\,.$$
\end{proposition}

\begin{proof}
First, we observe for any $0<s<t$ that 
$$\frac{\P(S>s)}{\P(S>t)}=\left(\frac{\P(T_1>s)}{\P(T_1>t)}\right)^{A-2}\to\infty\quad\text{and}\quad \P(s<S<t)\sim \P(S>s)\quad\text{for }A\to\infty\,. $$
Hence, we get for any positive random variable $T$, which is independent of $S$, that
\begin{equation}\label{eq a1}
    \P(s<S<T)=\int_s^\infty  \P(s<S<t)\,dP_{T}(t)\sim \P(S>s)\,\P(T>s)\quad\text{for }A\to\infty\,,
\end{equation}
where $P_T$ denotes the law of $T$. As a consequence, we get for all $s\geq 0$
\begin{equation}\label{eq a2}
P_{A, 2}(S>s)=\frac{\P(S>s,\,S<T_1,\,S<T_2)}{\P(S<T_1,\,S<T_2)}\sim\frac{\P(s<T_1,\,s<T_2)}{\P(S<T_1,\,S<T_2)}\P(S>s)\quad\text{for }A\to\infty\,.
\end{equation}
Let $P_{(T_2, T_2(x_2))}$ be the joint law of  $(T_2,\, T_2(x_2))$ with $ T_2(x_2)$ as in (\ref{eq: T2(x)}). Now, we are prepared for the final calculation, using the dominated convergence theorem we get for all $s\geq 0$ and $x_2 >X_2 (0)$
    \begin{align*}
        &\frac{ \P_{A, 2}(S>s\,|\, X_2(\infty)\ge x_2)}{ \P_{A, 2}(S>s)}=\frac{\P(S>s,\,S>T_2(x_2),\, S<T_1,\,S<T_2)}{ \P_{A, 2}(S>s)\,\P(S>T_2(x_2),\, S<T_1,\,S<T_2)}\\
        &=\int_0^\infty \int_0^t\frac{\P(S>s,\,S>t(x_2),\, S<T_1,\,S<t) dP_{(T_2, T_2(x_2))}(t, t(x_2))}{ \P_{A, 2}(S>s)\,\P(S>T_2(x_2),\, S<T_1,\,S<T_2)}\\
        &\ge \int_0^\infty \int_0^t \frac{\P(S>s,\, S<T_1,\,S<t)}{ \P_{A, 2}(S>s)\,\P(S>T_2(x_2),\, S<T_1,\,S<T_2)}\cdot\mathds{1}_{\{t(x_2)<s\}}dP_{(T_2, T_2(x_2))}(t, t(x_2))\\
        &\sim \int_0^\infty \int_0^s \frac{\P(S<T_1,\,S<T_2)\,\P(T_1 \wedge t>s)}{\P(s<T_1, s<T_2)\,\P(S>T_2(x_2),\, S<T_1,\,S<T_2)}dP_{(T_2, T_2(x_2))}(t, t(x_2))\\
        &\asymp \frac{\P(S<T_1,\,S<T_2)}{\P(S>T_2(x_2),\, S<T_1,\,S<T_2)}\to\infty\quad\text{for }A\to\infty\ ,
    \end{align*}
    using the aymptotics \eqref{eq a1}, \eqref{eq a2} and $\displaystyle\int_0^\infty \int_0^s \P(T_1 \wedge t>s)dP_{(T_2, T_2(x_2))}(t, t(x_2)) \asymp\P(s<T_1, s<T_2)$, with $\P \big( T_2 (x_2)<s\big) \in (0,1)$. 
The second part of the claim follows as follows:
\begin{align*}
    &P_{A, 2}(S>s\,|\, X_2(\infty)\ge x_2)^{-1}=\frac{\P(S>T_2(x_2),\, S<T_1,\,S<T_2)}{\P(S>T_2(x_2),\, S<T_1,\,S<T_2, S>s)}\\
    &\ge \int_0^\infty \int_0^t \frac{\P(S>t(x_2),\, S<T_1,\,S<t)\mathds{1}_{\{t(x_s)<s\}}\, dP_{(T_2, T_2(x_2))}(t, t(x_2))}{\P(S>s)}\\
    &\sim\int_0^\infty \int_0^s \frac{ \P(S>t(x_2))}{\P(S>s)}\,\P(t(x_2)<T_1 \wedge t)\, dP_{(T_2, T_2(x_2))}(t, t(x_2))\xrightarrow{A\to\infty}\infty\ .
\end{align*}
\end{proof}

Hence, $S$ becomes stochastically larger by conditioning on $X_2(\infty)\ge x_2$ and this change is stronger for larger $A$. Since $S$ covers the dependence between both agents,  $X_1(\infty)=\Xi_1(S)$ also becomes stochastically larger by conditioning on $X_2(\infty)\ge x_2$ and the effect is increasing with $A$, causing an increasing dependence in the sense measured by $c(A, a)$. Be aware that Proposition \ref{prop: increasingDependence} holds for fixed (i.e. non-diverging) $x_2$. Nevertheless, the dependence of losers is rather weak for all $A$ as the tail decay of $X_i(\infty)$ is only affected by other agents up to a constant prefactor, which is however not negligible even in the limit $A\to\infty$. In general, the $X_i(\infty)$ are asymptotically independent for $A\to\infty$ in a trivial way (see Appendix \ref{sec: manyAgents}). This is consistent with the second part of Proposition \ref{prop: increasingDependence}, which implies for large $A$ that $X_1(\infty)=X_1(0)$ is likely even conditioned on $X_2(\infty)\ge x_2$. This is not a contradiction as the constant $c(A, a)$ only measures the tail dependence, where we basically already picked two agents who won many steps.\\

The slightly positive correlation of losers can also be numerically underlined as follows. For $F(k)=k^2$, we executed 100,000 simulations each with $A=3$ and $A=30$. The simulations were stopped when one agent exceeded a share of $0.99$. The Pearson correlation coefficient of $\log(X_1(\infty))$ and $\log(X_2(\infty))$ conditioned on $sMon_1^c\cap sMon_2^c$ amounts to $+0.020$ for $A=3$ and $+0.025$ for $A=30$.\\

Next, we analyse the dependence of $c(A, a)$ on $a$. For that, let us rephrase Theorem \ref{thm: LoserCorrelation} in the fully symmetric case as
$$ \P_{A, a}(X_1(\infty)>x_1\,|\,X_2(\infty)>x_2,\ldots, X_a(\infty)>x_a)\sim\frac{c(A, a)}{c(A, a-1)} \P_{A, a}(X_1(\infty)>x_1)$$
for $x_1,\ldots, x_a\to\infty$. Note that $c(A, 1)=1$ by definition. Obviously, the information that several losers won in many steps is a stronger hint for a late explosion of the winner than having this information for only a few agents. Hence, we conjure that $c(A, a)/(A, a-1)$ is increasing in $a$ and, equivalently, $c(A, a)$ increases super-exponentially in $a$. The following table shows the ratio $c(A, a)/(A, a-1)$ for $F(k)=k^2$ and different $A, a$.

\begin{center}
\begin{tabular}{ c | c c c}
$\frac{c(A, a)}{c(A, a-1)}$ & $A=10^3$ & $A=10^6$ & $A=10^9$ \\\hline
 $a=3$  & 2.04 & 2.45 & 2.61 \\
 $a=10$ & 4.37 & 6.65 & 7.63 \\
 $a=20$ & 6.45 & 11.78 & 14.14 \\
 $a=30$ & 8.25 & 16.42 & 20.27 
\end{tabular}
\end{center}

These computations combined with the heuristic below motivate the following conjecture for large systems.

\begin{conjecture}\label{conj: c}
    For $F_1=\ldots=F_A$, $X(0)=(1,\ldots, 1)$ and any $a\geq 2$, we have 
    $$\lim_{A\to\infty}c(A, a)=a!\quad\mbox{so that}\quad \lim_{A\to\infty}\frac{c(A, a)}{c(A, a-1)} =a\ .$$
\end{conjecture}

The precise value $a!$ is justified by the following heuristic argument. First, we observe that $c(A, a)\sim\frac{\E [g(S)^a]}{\E[g(S)]^a}$ for $A\to\infty$ since $S\xrightarrow{A\to\infty}0$ almost surely. Now, we approximate for simplicity
\begin{equation}\label{eq: assumptionT}
    \P(T_i\le s)\approx cs^k\quad\text{for }s\to0
\end{equation}
for some $c>0$ and $k$ large (see Lemma \ref{lemma: gProperties}). By standard extreme value theory, we know that $\sqrt[k]{A}S$ converges for $A\to\infty$ towards a Weibull distributed random variable $W$ with density $w(s)=cks^{k-1}e^{-cs^k}$. Then we get for $A\to\infty$:
\begin{align*}
    \frac{\E [g(S)^a]}{\E[g(S)]^a}\sim \frac{\E [g(W/\sqrt[k]{A})^a]}{\E[g(W/\sqrt[k]{A})]^a}\sim\frac{\E\left[(cW^k/A)^a\right]}{\E[cW^k/A]^a}=\frac{\E[W^{ak}]}{\E[W^k]^a}=\frac{c^{-a}\Gamma(1+a)}{c^{-a}\Gamma(2)^a}=a!
\end{align*}
Since the limit does neither depend on $k$ nor on $c$, it stands to reason that the approximation (\ref{eq: assumptionT}) does not change the result. If conjecture \ref{conj: c} holds, then $c(A, a)/c(A, a-1)\to a>1$, underlining that the positive tail dependence of losers does not vanish in the limit $A\to\infty$.

\section{The time of monopoly}

Another interesting consequence of Theorem \ref{thm: loserSuperlin} and Theorem \ref{thm: loserSublin} concerns the time of monopoly as defined in (\ref{eq: MonTimeDef}). In the symmetric, super-linear two-agent case, this quantity has already been studied in \cite{Zhu, Cotar}, but with a slightly different definition (which differs from ours at most up to a multiplicative constant).  They found that the tail of $N_{mon}$ is heavier than the tail of $\min(X_1(\infty), X_2(\infty))$, which implies that it is not untypical for losers to win a few late steps, when the advantage of the winner is already big. In this chapter, we extend their results to the (mostly homogeneous) asymmetric case, where we also allow for sub-linear feedback.

\subsection{Time of monopoly for super-linear agents}\label{sec: time}

We directly formulate our main result for a system with two super-linear agents, who can both be the monopolist with positive probability. We will return to $A>2$ below.

\begin{corollary}\label{cor: MonTimeSuperlin}
    Let $A=2$ and $F_i(k)=\alpha_ik^{\beta_i}$ with $\beta_i>1$ and $\alpha_i>0$. Define $\beta\coloneqq {\frac{\beta_1-1}{\beta_2}}$ and assume that (\ref{eq: condMonTime}) holds.

\begin{enumerate}
    \item If $\beta_1\le\beta_2+1$, then we have
    $$\P(N_{mon}=n\,|\,sMon_1)\asymp n^{\beta-\beta_1}\quad\text{for }n\to\infty\,.$$
     \item If $\beta_1\ge\beta_2+1$, then we have
    $$\P(N_{mon}=n\,|\,sMon_1)\asymp n^{-\beta_2}\quad\text{for }n\to\infty\,.$$
\end{enumerate}
\end{corollary}

In the proof, $k$ will denote the number of steps won by agent $2$, when the monopoly of agent 1 set in at time $n$. We will have to distinguish small, moderate and large $k$ for given $n$. For that, we define
$$k_1(n)\coloneqq\begin{cases}
    n^{\beta} &\text{for }\beta_1<\beta_2+1\\
    \frac n2  &\text{for }\beta_1=\beta_2+1\\
    n-n^{\beta_2/(\beta_1-1)} &\text{for }\beta_1>\beta_2+1
\end{cases}$$
and
$$ k_2(n)\coloneqq \begin{cases}
    const. n^{(\beta_1-1)/(\beta_2-1)} &\text{for }\beta_1<\beta_2\\
    const. \frac n2 &\text{for }\beta_2\le\beta_1<\beta_2+1\\
    n-const. n^{\beta_2/\beta_1} &\text{for }\beta_1\ge\beta_2+1
\end{cases}$$
with $const.>0$, such that (\ref{eq: MonTimeConcentration}) holds. Note that $k_2 (n)\geq k_1 (n)$ for all $n$. In order to apply the dominated convergence theorem in the proof of Corollary \ref{cor: MonTimeSuperlin}, we need to require the following uniformity of the convergence $X_2(n)\to X_2(\infty)$:
\begin{equation}\label{eq: condMonTime}
    \sup_{n\in\N\atop k_1(n)\le k\le k_2(n) }\frac{\P(X_2(n)=X_2(0)+k,\,sMon_1)}{\P(X_2(\infty)=X_2(0)+k)}<\infty
\end{equation}
This assumptions seems plausible since the convergence $X_2(n)\nearrow X_2(\infty)$ is monotone. Nevertheless, a formal proof is still pending. This specific choice of $k_1(n)$ and $k_2(n)$ ensures that
$$\liminf_{n\to\infty}\frac{(n-k_2(n))^{\beta_1}}{k_2(n)^{\beta_2}}>0,\quad \limsup_{n\to\infty}\frac{k_1(n)^{\beta_2}}{(n-k_1(n))^{\beta_1-1}}<\infty,\quad\liminf_{n\to\infty}\frac{k_2(n)^{\beta_2}}{(n-k_2(n))^{\beta_1-1}}=\infty$$
holds, which we are going to use in the proof of Corollary \ref{cor: MonTime} extensively. We start the proof with a short lemma.

\begin{lemma}\label{lemma: monTime}
    Let $A=2$ and $F_i(k)=\alpha_ik^{\beta_i}$ with $\beta_i>1$ and $\alpha_i>0$. Then:
    $$\liminf_{n\to\infty}\inf_{ k_1(n)\le k\le k_2(n)}\frac{\P(X_2(n)=X_2(0)+k,\,sMon_1)}{\P(X_2(n)=X_2(0)+k)}>0$$
\end{lemma}

\begin{proof}
    We have to show that 
    $$\P(sMon_1\,|\,X_2(n)=X_2(0)+k)\ge\P(sMon_1\,|\,X_2(n)=X_2(0)+k_2(n))$$
    is bounded away from zero. Using the concentration property of explosion times around their expectation from \cite{wir}, it suffices to show that
    \begin{align}\label{eq: MonTimeConcentration}
        \liminf_{n\to\infty}\frac{\sum_{l=X_2(0)+k_2(n)}^{\infty}\frac{1}{F_2(l)}}{\sum_{l=X_1(0)+n-k_2(n)}^{\infty}\frac{1}{F_1(l)}}=\lim_{n\to\infty}\frac{\alpha_1(\beta_2-1)}{\alpha_2(\beta_1-1)}\cdot\frac{k_2(n)^{1-\beta_2}}{(n-k_2(n))^{1-\beta_1}}>1\,,
    \end{align}
    which is ensured by the definition of $k_2(n)$ in each case.
\end{proof}

\begin{proof}[Proof of Corollary \ref{cor: MonTimeSuperlin}]
    First, we express the desired probability as
    \begin{align}\label{eq: timeMonopoly}
        &\P(N_{mon}=n+1, sMon_1)\\
        &\quad=\sum_{k=1}^{n-1}\P(X_2(n-1)=k+X_2(0))\cdot\P(N_{mon}=n+1, sMon_1\,|\,X_2(n-1)=k+X_2(0))\nonumber\\
        &\quad=\sum_{k=1}^{n-1}\P(X_2(n-1)=k+X_2(0))\cdot\frac{F_2(X_2(0)+k)}{F_1(X_1(0)+n-1-k)+F_2(X_2(0)+k)}\cdot R(k, n)\,,\nonumber
    \end{align}
    where 
    \begin{equation}\label{eq: R(k,n)}
        R(k, n)\coloneqq\prod_{l=0}^\infty \frac{F_1(X_1(0)+n-k-1+l)}{F_1(X_1(0)+n-k-1+l)+F_2(X_2(0)+k+1)}<1 
    \end{equation}
    is the probability, that agent 1 wins all steps after step $n$, when agent 2 won $k+1$ of the first $n$ steps. Due to $\frac{1}{1+x}\ge e^{-x}$ for all $x>-1$, we have the following lower bound for $R(k, n)$:
    \begin{align*}
        R(k,n)\ge \exp\left(-F_2(X_2(0)+k+1)\sum_{l=X(0)+n-k-1}^\infty\frac{1}{F_1(l)}\right)\ge \exp\left(-const.k^{\beta_2}(n-k)^{1-\beta_1}\right)
    \end{align*}

    1. Now assume that $\beta_1-\beta_2<1$, such that $\beta<1$ and $R(k, n)$ is uniformly bounded away from zero as long as $k\le k_1(n)=n^\beta$. Using Theorem \ref{thm: loserSuperlin}, we get the following asymptotic for the first $k_1(n)$ summands in (\ref{eq: timeMonopoly}): 
    \begin{align}\label{eq: MonTimeSum1}
        &\sum_{k=1}^{k_1(n)}\P(X_2(n-1)=k+X_2(0))\cdot\frac{F_2(X_2(0)+k)}{F_1(X_1(0)+n-1-k)+F_2(X_2(0)+k)}\cdot R(k, n)\nonumber\\
        &\asymp \sum_{k=1}^{n^\beta}\frac{\P(X_2(n-1)=k+X_2(0))}{\P(X_2(\infty)=k+X_2(0))}\cdot\frac{1}{F_1(n-k)}\asymp n^{-\beta_1} \sum_{k=1}^{n^\beta}\frac{\P(X_2(n-1)=k+X_2(0))}{\P(X_2(\infty)=k+X_2(0))}\nonumber\\
        &=n^{-\beta_1}\int_0^{n^\beta}\frac{\P(X_2(n-1)=\lceil u\rceil+X_2(0))}{\P(X_2(\infty)=\lceil u\rceil+X_2(0))}\,du=n^{\beta-\beta_1}\int_{0}^1\frac{\P(X_2(n-1)=\lceil n^\beta u\rceil+X_2(0))}{\P(X_2(\infty)=\lceil n^\beta u\rceil+X_2(0))}\,du\nonumber\\
        &\sim n^{\beta-\beta_1}\quad\text{for }n\to\infty
    \end{align}
    In the last line, we applied the dominated convergence theorem, which holds since
    \begin{align}\label{eq: MonTimeMajorante}
        &\frac{\P(X_2(n-1)=\lceil n^\beta u\rceil+X_2(0))}{\P(X_2(\infty)=\lceil n^\beta u\rceil+X_2(0))}\le \frac{\P(X_2(n-1)=\lceil n^\beta u\rceil+X_2(0))}{\P(X_2(\infty)=X_2(n-1)=\lceil n^\beta u\rceil+X_2(0))}\nonumber\\
        &=\frac{1}{\P(X_2(\infty)=X_2(n-1)\,|\,X_2(n-1)=\lceil n^\beta u\rceil+X_2(0))}=\frac{1}{R(\lceil n^\beta u\rceil-1, n-1)}\nonumber\\
        &\le \frac{1}{R(\lceil n^\beta \rceil-1, n-1)}
    \end{align}
    
    To complete the proof of 1., it remains to show that the remaining summands in (\ref{eq: timeMonopoly}) are at most of the same order. For that,  we first observe that for $k\le k_2(n)$
   \begin{align}\label{eq: upperBoundR}
        R(k,n)\le \exp\left(-const.F_2(X_2(0)+k+1)\sum_{l=X(0)+n-k-1}^\infty\frac{1}{F_1(l)}\right)\le \exp\left(-const.k^{\beta_2}(n-k)^{1-\beta_1}\right)
    \end{align}
    holds due to $\frac{1}{1+x}\le e^{-cx}$ for $c\le\frac{\log(1+x)}{x}$ and $x>0$. Now, we have to split the sum in (\ref{eq: timeMonopoly}) once again. Now, we need Lemma \ref{lemma: monTime} and (\ref{eq: condMonTime}).
    \begin{align}\label{eq: monTimeSum2}
        &\sum_{k=k_1(n)+1}^{k_2(n)}\P(X_2(n-1)=k+X_2(0))\cdot\frac{F_2(X_2(0)+k)}{F_1(X_1(0)+n-1-k)+F_2(X_2(0)+k)}\cdot R(k, n)\\
        &\prec\sum_{k=n^\beta+1}^{k_2(n)}\frac{1}{F_1(n)}\exp\left(-const.k^{\beta_2}(n-k)^{1-\beta_1}\right)
        \prec n^{-\beta_1}\sum_{k=n^\beta+1}^{k_2(n)}\exp\left(-const.k^{\beta_2}n^{1-\beta_1}\right)\nonumber\\
        &\le n^{-\beta_1}\int_{n^\beta}^{k_2(n)}\exp\left(-const.u^{\beta_2}n^{1-\beta_1}\right)du= n^{-\beta_1}n^\beta\int_1^{k_2(n)n^{-\beta}}\exp\left(-const.u^{\beta_2}\right)du\asymp n^{-\beta_1}n^\beta\nonumber
    \end{align}
    For the remaining summands in (\ref{eq: timeMonopoly}), a rough estimate is sufficient:
     \begin{align}\label{eq: monTimeSum3}
        &\sum_{k=k_2(n)+1}^{n-1}\P(X_2(n-1)=k+X_2(0))\cdot\frac{F_2(X_2(0)+k)}{F_1(X_1(0)+n-1-k)+F_2(X_2(0)+k)}\cdot R(k, n)\nonumber\\
        &\le\sum_{k=k_2(n)+1}^{n-1} R(k, n)\le\sum_{k=k_2(n)+1}^{n-1}R(k_2(n), n)\le (n-k_2(n))\exp\left(-const.k_2(n)^{\beta_2}n^{1-\beta_1}\right)
    \end{align}
    Obviously, the last term decays faster than any polynomial for $\beta_1-\beta_2<1$. 

    2. Let $\beta_1-\beta_2\ge1$. Since $N_{mon}\ge X_2(\infty)-X_2(0)$ on the event $sMon_1$, an appropriate asymptotic lower bound follows directly from Theorem \ref{thm: loserSuperlin}. For an upper bound, define $\beta'=\frac{\beta_2}{\beta_1-1}$. Widely analogous to part 1, we have to split the sum (\ref{eq: timeMonopoly}) in three parts again. Suppose for a moment $\beta_1-\beta_2>1$, such that $\beta'<1$. Then:
    \begin{align*}
        &\sum_{k=1}^{k_1(n)}\P(X_2(n-1)=k+X_2(0))\cdot\frac{F_2(X_2(0)+k)}{F_1(X_1(0)+n-1-k)+F_2(X_2(0)+k)}\cdot R(k, n)\nonumber\\
        &\prec \sum_{k=1}^{n-n^{\beta'}}\frac{1}{F_1(n-k)}= \sum_{k=n^{\beta'}}^{n-1}\frac{1}{F_1(k)}\asymp \left(n^{\beta'}\right)^{1-\beta_1}=n^{-\beta_2}
    \end{align*}
    In the first step, we used (\ref{eq: MonTimeMajorante}), Theorem \ref{thm: loserSuperlin} and $R(k, n)<1$. Similarly to (\ref{eq: monTimeSum2}), the second partial sum is:
     \begin{align*}
        &\sum_{k=k_1(n)+1}^{k_2(n)}\P(X_2(n-1)=k+X_2(0))\cdot\frac{F_2(X_2(0)+k)}{F_1(X_1(0)+n-1-k)+F_2(X_2(0)+k)}\cdot R(k, n)\\
        &\prec\sum_{k=n-n^{\beta'}+1}^{n-n^{\beta_2/\beta_1}}\frac{1}{F_1(n-k)}\exp\left(-const.k^{\beta_2}(n-k)^{1-\beta_1}\right)\\
        &\prec\int_{n-n^{\beta'}}^{n-n^{\beta_2/\beta_1}}(n-u)^{-\beta_1}\exp\left(-const.u^{\beta_2}(n-u)^{1-\beta_1}\right)du\\
        &=\int_{n^{\beta_2/\beta_1}}^{n^{\beta'}}u^{-\beta_1}\exp\left(-const.(n-u)^{\beta_2}u^{1-\beta_1}\right)du\prec\int_{0}^{n^{\beta'}}u^{-\beta_1}\exp\left(-const.n^{\beta_2}u^{1-\beta_1}\right)du\\
        &=\int_{0}^{1}n^{\beta'-\beta_1\beta'}u^{-\beta_1}\exp\left(-const.u^{1-\beta_1}\right)du=n^{-\beta_2}\int_1^\infty u^{\beta_1-2}\exp\left(-const.u^{\beta_1-1}\right)du
    \end{align*}
    In the second line, we used assumption (\ref{eq: condMonTime}). The remaining summands decay exponentially fast like in (\ref{eq: monTimeSum3}). For the case $\beta_1-\beta_2=1$, the proof is completely analogous with the given choice of $k_1$ and $k_2$.
\end{proof}

 Without assumption (\ref{eq: condMonTime}), Corollary \ref{cor: MonTimeSuperlin} still holds if $\asymp$ is replaced by $\succ$ due to first partial sum (\ref{eq: MonTimeSum1}). By using Lemma \ref{lemma: monTime} and the rough estimate
 \begin{align}\label{eq: MonTimeUpperBound}
     \P(X_2(n-1)=k+X_2(0))&\le \P(X_2(n-1)\ge k+X_2(0))\\
     &\asymp\P(X_2(n-1)\ge k+X_2(0), sMon_1)\le \P(X_2(\infty)\ge k+X_2(0), sMon_1)\nonumber
 \end{align}
 in the middle sum (\ref{eq: monTimeSum2}), we obtain the upper bounds
 $$\P(N_{mon}=n\,|\,sMon_1)\prec n^{2\beta-\beta_1}\quad\text{for }\beta_1\le\beta_2+1\,,$$
and
$$\P(N_{mon}=n\,|\,sMon_1)\prec n^{1-\beta_2}\quad\text{for }\beta_1\ge\beta_2+1$$
  without assumption (\ref{eq: condMonTime}). This upper bound is particularly good for large $\beta_2$.\\

Moreover, it is possible to extend  Corollary \ref{cor: MonTimeSuperlin} to $A>2$ via the exponential embedding.

\begin{corollary}
    Let $F_i(k)=\alpha_ik^{\beta_i}$ with $\beta_i>1$ and $\alpha_i>0$, $i\in[A]$. Assume $\beta_2\le \beta_j$ for all $j\ge2$ and that (\ref{eq: condMonTime}) holds for any two agent system with feedback $F_1, F_j$. Define $\beta\coloneqq {\frac{\beta_1-1}{\beta_2}}$.

\begin{enumerate}
    \item If $\beta_1\le\beta_2+1$, then we have
    $$\P(N_{mon}>n\,|\,sMon_1)\asymp n^{\beta-\beta_1+1}\quad\text{for }n\to\infty\,.$$
     \item If $\beta_1\ge\beta_2+1$, then we have
    $$\P(N_{mon}>n\,|\,sMon_1)\asymp n^{-\beta_2+1}\quad\text{for }n\to\infty\,.$$
\end{enumerate}
\end{corollary}

\begin{proof}
    First, a lower bound follows from Corollary \ref{cor: MonTimeSuperlin} via canonical coupling. In the exponential embedding, denote by $s_i\ge0$, $i\ge 2$ the time of the last jump of $\Xi_i$ before the explosion time $T_1$ of agent 1. Then, $N_{mon}$ can be expressed as
    $$N_{mon}=\Xi_1\left(\max_{i=2,\ldots, A}s_i\right)+X_2(\infty)+\ldots X_A(\infty)+1$$
    on the event $sMon_1$. Hence:
    \begin{align*}
        &\P(N_{mon}>n\,|\,sMon_1)\le \frac{1}{\P(sMon_1)}\P\left(\sum_{i=2}^A (X_i(\infty)+\Xi_1(s_i))\ge n,\, sMon_1\right)\\
        &\le \frac{1}{\P(sMon_1)}\sum_{i=2}^A\P\left( (X_i(\infty)+\Xi_1(s_i))\ge \frac{n}{A-1},\, T_i>T_1\right)
    \end{align*}
    These probabilities are covered for $n\to\infty$ by Corollary \ref{cor: MonTimeSuperlin} as they correspond to the probability of $N_{mon}>\frac{n}{A+1}$ in a two agent system. Note that the summand $i=2$ dominates the others due to the assumption  $\beta_2\le \beta_j$.
\end{proof}

Hence, in large systems only the loser with the weakest feedback determines the tail of $N_{mon}$.

Note that the tail weight of the time on monopoly does also depend on the feedback of the winner, as opposed to the wealth of the loser discussed in Theorem \ref{thm: loserSuperlin}. As a consequence of Corollary \ref{cor: MonTimeSuperlin}, the tail of $N_{mon}$ (without conditioning on a winner) is given by $\P(N_{mon}=n)\asymp n^{\beta-\beta_1}$, where w.l.o.g. $\beta_1\le\beta_2$ and $A=2$. For the symmetric case $\beta_1=\beta_2$, this is consistent with \cite{Zhu, Cotar}. According to  Corollary \ref{cor: MonTimeSuperlin}, the (unconditioned) tail of $N_{mon}$ is heavier than the one of $\min\{X_1(\infty), X_2(\infty)\}$ in any case. In the conditioned situation, if the feedback of the winner is much stronger than the one of the loser, then $N_{mon}$ is of the same order as the wealth of the loser, i.e. the loser basically only wins steps at the beginning of the process. If the feedback of the loser is at most slightly stronger, then  $N_{mon}$ is of higher order than the wealth of the loser and, hence, the loser might win some late steps when the advantage of the winner is already large. Let us formalize this idea.

\begin{corollary}\label{cor: ShareMonTime}
      Let $A=2$ and $F_i(k)=\alpha_ik^{\beta_i}$ with $\beta_i>1$ and $\alpha_i>0$. Assume that (\ref{eq: condMonTime}) holds. Then:
      \begin{enumerate}
          \item If $\beta_1<\beta_2+1$, then
          $$\forall \epsilon>0\colon \lim_{n\to\infty}\P(\chi_2(N_{mon})<\epsilon\,|\, N_{mon}=n, sMon_1)=1\,.$$
            \item If $\beta_1=\beta_2+1$, then
          $$\forall \delta>0\ \exists\epsilon>0\colon \limsup_{n\to\infty}\P(\chi_2(N_{mon})<\epsilon\,|\, N_{mon}=n, sMon_1)<\delta$$ 
          and
          $$\forall \epsilon>0\colon \liminf_{n\to\infty}\P(\chi_2(N_{mon})>1-\epsilon\,|\, N_{mon}=n, sMon_1)>0\,.$$
          \item If $\beta_1>\beta_2+1$, then
           $$\forall \epsilon>0\colon \lim_{n\to\infty}\P(\chi_2(N_{mon})>1-\epsilon\,|\, N_{mon}=n, sMon_1)=1\,.$$
      \end{enumerate}
\end{corollary}

\begin{proof}
    Recall the proof of Corollary \ref{cor: MonTime}. In particular, we showed that the sum (\ref{eq: timeMonopoly}) is dominated by the summands $k=1,\ldots, k_2(n)$, where $k$ represents the number of steps won by agent 1, when $N_{mon}=n$ and $sMon_1$. If $\beta_1<\beta_2+1$, then the summands $k=1,\ldots, \epsilon n$ dominate the whole sum for any $\epsilon>0$, which implies 1.. In contrast, if $\beta_1=\beta_2+1$, then the summands $k=\epsilon n,\ldots, n$ are of the same order as the whole sum and for any sequence $(l_n)_n$ with $l_n/n\to 0$ the summands $k=l_n,\ldots, n$ dominate the whole sum. Hence, 2. holds. Finally for $\beta_1>\beta_2+1$, the summands $k=1,\ldots, (1-\epsilon)n$ are of order $n^{1-\beta_1}<n^{-\beta_2}$, such that 3. follows.
\end{proof}

In other words, if the monopoly sets in late, then this typically happens in two different ways depending on the feedback functions: First, if the feedback of the loser is weak compared to the winner, then the loser wins a significant share of steps until the monopoly suddenly set in. For $\beta_1>\beta_2+1$, the loser's share at time $N_{mon}$ is even close to one, if the monopoly sets in late. Second, if the feedback of the loser is strong enough and the monopoly sets in late, then the loser can win steps at late time, when the winner was already dominant.\\

  In the situation of 1. in Corollary \ref{cor: MonTimeSuperlin}, the tail of $N_{mon}$ is heavier the weaker the feedback of both agents. This appears intuitive for the winner, but also weak feedback of the loser increases the probability of a late occurrence of monopoly. This surprising fact corresponds to the loser-paradox explained in Section \ref{sec: applicatonPolya}. The difference between the tail weight of the loser's wealth and of $N_{mon}$ is largest for strong feedback of the loser, since the tail of $X_2(\infty)$ is lighter. These findings are underlined by Figure \ref{figure: LoserSimulation} (a). Note that Corollary \ref{cor: MonTimeSuperlin} is consistent at the transition point $\beta_1=\beta_2+1$. This transition does also occur in the context of total monopoly as discussed in \cite{wir}.\\

For symmetric, exponentially increasing feedback, the tail of $N_{mon}$ decays exponentially as shown in \cite{Cotar}. Let us finally take a quick look at mixed feedback, say $F_1(k)=k^{\beta_1}$ and $F_2 (k)=e^{\beta_2 k}$ with $\beta_1>1$ and $\beta_2>0$. Of course, we still have the trivial lower bounds
$$\P(N_{mon}=n\,|\,sMon_2)\succ n^{-\beta_1}\quad\text{and}\quad\P(N_{mon}=n\,|\,sMon_1)\succ e^{-\beta_2 n}  $$
due to Theorem \ref{thm: loserSuperlin}. Extending 2. of Corollary \ref{cor: MonTimeSuperlin} to the limit $\beta_2\to\infty$, it stands to reason that 
$$\P(N_{mon}=n\,|\,sMon_2)\asymp n^{-\beta_1}\,.$$
The case $sMon_1$ behaves rather surprisingly.

\begin{proposition}
    Let $F_1(k)=k^{\beta_1}$ and $F_2 (k)=e^{\beta_2 k}$ with $\beta_1>1$ and $\beta_2>0$. Then:
    $$\P(N_{mon}=n\,|\,sMon_1)\asymp n^{-\beta_1}\log(n)$$
\end{proposition}

\begin{proof}
    Analogously to the proof of Corollary \ref{cor: MonTime}, we consider the sum (\ref{eq: timeMonopoly}) and define $k_1(n)=\log(n^{\beta_1-1})/\beta_2$ and $k_2(n)=\log(n^{\beta_1})/\beta_2$. For the first partial sum, we use that $R(k, n)$ (defined mutatis mutandis to (\ref{eq: R(k,n)})) is bounded from below.
    \begin{align*}
        &\sum_{k=1}^{k_1(n)}\P(X_2(n-1)=k+X_2(0))\cdot\frac{F_2(X_1(0)+k)}{F_1(X_2(0)+n-1-k)+F_2(X_1(0)+k)}\cdot R(k, n)\nonumber\\
        &\asymp \sum_{k=1}^{\log(n^{\beta_1-1})/\beta_2}\frac{1}{(n-k)^{\beta_1}}\sim n^{-\beta_1}\log(n^{\beta_1-1})/\beta_2\quad\text{for }n\to\infty
    \end{align*}
    The dominated convergence theorem is applicable by an analogous argument as (\ref{eq: MonTimeMajorante}). Repeating the idea of (\ref{eq: MonTimeUpperBound}), we get for the second partial sum:
    \begin{align*}
        &\sum_{k=k_1(n)}^{k_2(n)}\P(X_2(n-1)=k+X_2(0))\cdot\frac{F_2(X_1(0)+k)}{F_1(X_2(0)+n-1-k)+F_2(X_1(0)+k)}\cdot R(k, n)\nonumber\\
        &\prec \sum_{\log(n^{\beta_1-1})/\beta_2}^{\log(n^{\beta_1})/\beta_2}\frac{k}{(n-k)^{\beta_1}}e^{-const. e^{k}(n-k)^{1-\beta_1}}\asymp n^{-\beta_1}\int_{\log(n^{\beta_1-1})/\beta_2}^{\log(n^{\beta_1})/\beta_2} u e^{-const. e^{u}n^{1-\beta_1}}du\\
        &\prec n^{-\beta_1}\int_1^\infty\frac{\log(n^{\beta_1-1}u)}{u}e^{-const. u}du\asymp n^{-\beta_1}\log(n^{\beta_1-1})\quad\text{for }n\to\infty
    \end{align*}
    The remaining summands decay exponentially like in (\ref{eq: monTimeSum3}).
\end{proof}

Basically, this is consistent with the limit $\beta_1\to\infty$ in part 1. of Corollary \ref{cor: MonTimeSuperlin}, but the additional term $\log(n)$ was not predicted. In particular in the case $sMon_1$, the time of monopoly is heavy tailed, although the wealth of the loser is light tailed. Surprisingly, the tail of $N_{mon}$ is heavier on $sMon_1$ than on $sMon_2$. Overall, the (unconditioned) tail of $N_{mon}$ is of order $\log(n)$ heavier than the tail of $\min\{X_1(\infty), X_2(\infty)\}$. Figure \ref{figure: LoserSimulation} (b) illustrates this situation. On $sMon_2$, the tails of $X_1(\infty)$ and $N_{mon}$ even seem to be equal,  i.e. the loser wins most steps before time $N_{mon}$ and 3. of Corollary \ref{cor: ShareMonTime} holds analogously.

\subsection{Time of monopoly for sub-linear agents}\label{sec: timeSublin}

The time of monopoly $N_{mon}$ (see (\ref{eq: MonTimeDef})) is also well-defined when only one agents satisfies the monopoly condition (\ref{eq: explosion}), such that this agent is almost surely the strong monopolist (see Section \ref{sec: LoserSublin}). The following corollary establishes the tail behaviour of $N_{mon}$ in this situation.

\begin{corollary}\label{cor: MonTime}
Let $A=2$ and $F_i(k)=\alpha_i k^{\beta_i}$ with $\beta_1\le 1<\beta_2$ and $\alpha_i>0$. If $\beta_1=1$, assume additionally $X_2(0)>\beta_2$. Then:

\begin{align}\label{eq: MonTime}
    \P(N_{mon}>n)\sim \E F_1(X_1(\infty))\cdot\sum_{k=n}^\infty\frac{1}{F_2(k)}\quad\text{for }n\to\infty
\end{align}
\end{corollary}

Note that $\E F_1(X_1(\infty))<\infty$ is always satisfied for $\beta_1<1$ and is equivalent to $X_2(0)>1$ for $\beta=1$ due to Theorem \ref{thm: loserSublin}.


\begin{proof}
    The probability of $N_{mon}=n+1$ can be expressed as
    \begin{align*}
        &\P(N_{mon}=n+1)=\sum_{k=0}^{n-1}\P\left(X_1(n-1)-X(0)=k\right)\\
        &\quad\cdot\P\left(X(n)-X(n-1)=e^{(1)},\, \forall m>n\colon X(m)-X(m-1)=e^{(2)}\,\big|\,X_1(n-1)-X(0)=k \right)
    \end{align*}
    Since for all $k\ge0$
    $$\lim_{n\to\infty}\P\left(X_1(n-1)-X(0)=k\right)=\P\left(X_1(\infty)-X(0)=k\right)$$
    and
    \begin{align}\label{eq: TimSublinAsymp}
        &\P\left(X(n)-X(n-1)=e^{(1)},\, \forall m>n\colon X(m)-X(m-1)=e^{(2)}\,\big|\,X_1(n-1)-X(0)=k \right)\nonumber\\
        &=\frac{F_1(X_1(0)+k)}{F_1(X_1(0)+k)+F_2(X_2(0)+n-k-1)}\cdot \prod_{m>n}\frac{F_2(X_1(0)-2-k+m)}{F_1(X_1(0)+k+1)+F_2(X_2(0)-k-2+m)}\nonumber\\
        &\sim \frac{F_1(X_1(0)+k)}{F_2(n)}\quad\text{for }n\to\infty
    \end{align}
    holds, we get the following asymptotic lower bound for $n\to\infty$ via Fatou's Lemma:
    \begin{align*}
        \P(N_{mon}=n+1)&\succ  \left(\sum_{k=0}^\infty\P\left(X_1(\infty)-X(0)=k\right)F_1(X_1(0)+k)\right)\frac{1}{F_2(n)}=\frac{\E F_1(X_1(\infty))}{F_2(n)}
    \end{align*}
    Due to non-uniformity of the summands, these considerations are not sufficient for an upper bound, but instead we can split the sum as follows. Take $\epsilon\in(0, 1)$, then:
    \begin{align*}
        &\P(N_{mon}=n+1)\le\P(X_1(\infty)-X_1(0)>\lfloor \epsilon n\rfloor)+\sum_{k=0}^{\lfloor \epsilon n\rfloor}\P\left(X_1(n-1)-X(0)=k\right)\\
        &\qquad\cdot\frac{F_1(X_1(0)+k)}{F_1(X_1(0)+k)+F_2(X_2(0)+n-k-1)} \prod_{m>n}\frac{F_2(X_1(0)-2-k+m)}{F_1(X_1(0)+k+1)+F_2(X_2(0)-k-2+m)}\\
        &\le o(n^{-\beta_2})+\sum_{k=0}^{\lfloor \epsilon n \rfloor}\P\left(X_1(n-1)-X(0)=k\right)\cdot \frac{F_1(X_1(0)+k)}{F_1(X_1(0)+1)+F_2(X_2(0)+n-\lfloor\epsilon n\rfloor-1)}\\
        &\le o(n^{-\beta_2})+\frac{ \E F_1(X_2(n-1))}{F_1(X_1(0)+1)+F_2(X_2(0)+n-\lfloor \epsilon n\rfloor-1)}\\
        &\sim \frac{1}{(1-\epsilon)^{\beta_2}}\cdot \frac{\E F_1(X_1(\infty))}{F_2(n)}\quad\text{for }n\to\infty
    \end{align*}
    In the second inequality, we used Theorem \ref{thm: loserSublin} and the assumption $X_2(0)>\beta_2$ in the case of $\beta_1=1$. $\epsilon\to 0$ finally yields the desired upper bound.
\end{proof}

By an analogous but more lengthy proof, one could generalize Corollary \ref{cor: MonTime} to $A>2$. Assume that there are several agents $i$ fulfilling the same assumptions as agent 1 in Corollary \ref{cor: MonTime}, but only one like agent 2. Then (\ref{eq: MonTime}) still holds when $\E F_1(X_1(\infty))$ is replaced by $\sum_i\E F_i(X_i(\infty))$. An analogous extension to exponential $F_2$ fails at step (\ref{eq: TimSublinAsymp}). Moreover for almost linear feedback, e.g. $F_1(k)=k\log(k)^{\beta_1}$, $\beta\in (0,1)$, the expectation $\E F_1(X_1(\infty))=\infty$ is infinite (see Example \ref{example: loserWealthSublin}), such that an analogous argument is not possible.\\\

As a consequence of Corollary \ref{cor: MonTime}, the time of monopoly $N_{mon}$ has a power law tail, although the tail of $X_1(\infty)$ is lighter than a power-law for $\beta_1<1$. Also for $\beta_1=1$, one can easily calculate that the tail of $N_{mon}$ is heavier than the one of $X_1(\infty)$. In contrast to $X_1(\infty)$, the tail of $N_{mon}$ does never depend on initial values (up to the constant prefactor). Remarkably in Corollary \ref{cor: MonTime}, the tail decay of $N_{mon}$ does only depend on the feedback of the winner up to a constant prefactor, as opposed to the super-linear case of Appendix \ref{sec: time}. As intuitively expected, the tail decays faster for strong feedback of the winner. All findings of this section are illustrated by Figure \ref{figure: LoserSimulation2}. Let us finally discuss an example to understand the transition between Corollary \ref{cor: MonTime} and Corollary \ref{cor: MonTimeSuperlin}.

\begin{example}
    Let $F_i(k)=k^{\beta_i}$ with $\beta_1>1$ and varying $\beta_2\in\R$. Then $\P(N_{mon}>x\,|\,sMon_1)\asymp x^{-a}$ for some exponent $a=a(\beta_2)>0$. This exponent is constant for $\beta_2\le 1$ and jumps downwards (even to zero if $\beta_1\ge2$) at $\beta_2=1$ in a left-continuous manner. For $\beta_2>1$, $a(\beta_2)$ is strictly increasing and with $\lim_{\beta_2\to\infty}a(\beta_2)=a(0)=1-\beta_1$. The tail of $N_{mon}$ coincides with the tail of $X_2(\infty)$ only for $1<\beta_2\le\beta_1-1$ (up to constant prefactors). The situation is illustrated by Figure \ref{figure: ExponentPlot}.
\end{example}

\end{document}